\documentclass[onesided,10pt]{amsart}
\usepackage{amsmath}
\usepackage{datetime}
\usepackage{ragged2e}
\usepackage{amssymb}
\usepackage[left=1.1in,right=1.1in]{geometry}

\newtheorem{theorem}{Theorem}
\newtheorem{lemma}[theorem]{Lemma}

\newtheorem{proposition}[theorem]{Proposition}
\newtheorem{definition}[theorem]{Definition}
\newtheorem{remark}[theorem]{Remark}

\newtheorem{corollary}[theorem]{Corollary}

\usepackage{amsmath}
\usepackage{amssymb}
\usepackage{amsfonts}
\usepackage{mathtools}
\usepackage{enumerate}
\usepackage{comment}
\usepackage{braket}

\usepackage{palatino}
\usepackage[T1]{fontenc}
\usepackage{dsfont} 

\usepackage{acronym}
\usepackage{latexsym}
\usepackage{paralist}
\usepackage{wasysym}
\usepackage{xspace}

\usepackage[dvipsnames,svgnames]{xcolor}
\colorlet{MyBlue}{DodgerBlue!60!Black}
\colorlet{MyGreen}{DarkGreen!85!Black}

\usepackage[font=small,labelfont=bf]{caption}
\usepackage{subfigure}
\usepackage{tikz}
\usetikzlibrary{calc,patterns}

\usepackage{hyperref}
\hypersetup{
colorlinks=true,
linktocpage=true,
pdfstartview=FitH,
breaklinks=true,
pdfpagemode=UseNone,
pageanchor=true,
pdfpagemode=UseOutlines,
plainpages=false,
bookmarksnumbered,
bookmarksopen=false,
bookmarksopenlevel=1,
hypertexnames=true,
pdfhighlight=/O,
urlcolor=MyBlue!60!black,linkcolor=MyBlue!70!black,citecolor=DarkGreen!70!black, 
pdftitle={},
pdfauthor={},
pdfsubject={},
pdfkeywords={},
pdfcreator={pdfLaTeX},
pdfproducer={LaTeX with hyperref}
}

\numberwithin{equation}{section}  
\usepackage[sort&compress,capitalize,nameinlink]{cleveref}
\crefname{app}{Appendix}{Appendices}

\crefrangeformat{equation}{\upshape(#3#1#4)\textendash(#5#2#6)}

\usepackage[textwidth=30mm]{todonotes}

\usepackage{soul}
\setstcolor{red}
\sethlcolor{SkyBlue}

\newcommand{\w}{\mathbf{w}}

\newcommand{\si}{\sigma} 
\newcommand{\ent}{{\rm ENT} }

\newcommand{\tc}{\, |\, }

\newcommand{\ind}{\mathbf{1}}
\newcommand{\p}{\mathfrak{p}}

\newcommand{\tv}{\texttt{TV}}
\newcommand{\bd}{\mathbf d}
\newcommand{\muin}{\mu_{\rm in\,\!}}

\newcommand{\tent}{T_\ent}
\newcommand{\bx}{\mathbf x}
\newcommand{\by}{\mathbf y}

\newcommand{\bo}{\mathbf o}

\let\a=\alpha    \let\d=\delta  \let\e=\varepsilon
 \let\g=\gamma     \let\k=\kappa  \let\l=\lambda
\let\o=\omega    
\let\r=\rho  \let\t=\tau  
\let\D=\Delta   \let\G=\Gamma

\newcommand{\cA}{\ensuremath{\mathcal A}} 
\newcommand{\cB}{\ensuremath{\mathcal B}} 
\newcommand{\cC}{\ensuremath{\mathcal C}} 
\newcommand{\cD}{\ensuremath{\mathcal D}} 
\newcommand{\cE}{\ensuremath{\mathcal E}} 
\newcommand{\cF}{\ensuremath{\mathcal F}} 
\newcommand{\cG}{\ensuremath{\mathcal G}} 
 
\newcommand{\cI}{\ensuremath{\mathcal I}}

\newcommand{\cP}{\ensuremath{\mathcal P}}

\newcommand{\cS}{\ensuremath{\mathcal S}} 
\newcommand{\cT}{\ensuremath{\mathcal T}}

\newcommand{\bbE}{{\ensuremath{\mathbb E}} }

\newcommand{\bbN}{{\ensuremath{\mathbb N}} } 
 
\newcommand{\bbP}{{\ensuremath{\mathbb P}} } 
 
\newcommand{\bbR}{{\ensuremath{\mathbb R}} }

\newcommand{\bbZ}{{\ensuremath{\mathbb Z}} } 

\newcommand{\N}{\ensuremath{\mathbb{N}}}

\renewcommand{\P}{\ensuremath{\mathbb{P}}}

\def\({\left(}
\def\){\right)}
\def\[{\left[}
\def\]{\right]}
\newacro{NE}{Nash equilibrium}
\newacroplural{NE}[NE]{Nash equilibria}
\newacro{PNE}{pure Nash equilibrium}
\newacroplural{PNE}[PNE]{Nash equilibria}
\newacro{PFNE}{prior-free Nash equilibrium}
\newacroplural{PFNE}[PFNE]{prior-free Nash equilibria}
\newacro{WE}{Wardrop equilibrium}
\newacroplural{WE}[WE]{Wardrop equilibria}
\newacro{SO}{socially optimum}

\newacro{KKT}{Karush\textendash Kuhn\textendash Tucker}
\newacro{OD}[O/D]{origin-destination}
\newacro{PoA}{price of anarchy}
\newacro{PoS}{price of stability}
\newacro{PoCS}{price of correlated stability}
\newacro{BPR}{bureau of public roads}
\newacro{FIP}{finite improvement property}

\newacro{BPG}{buck-passing game}
\newacro{SBPG}{stochastic buck-passing game}
\newacro{MBPG}{mixed extension of the buck-passing game}

\newacro{DCM}{directed configuration model}
\newacro{OCM}{out configuration model}

\begin{document}
\title
[Mixing time of PageRank surfers on sparse random digraphs]
{Mixing time of PageRank surfers\\ on sparse random digraphs%
}
\author[P.~Caputo]{Pietro Caputo$^{\sharp}$}
\address{$^{\sharp}$ Dipartimento di Matematica e Fisica, Universit\`a  Roma Tre, Largo Murialdo 1, 00146 Roma, Italy.}
\email{caputo@mat.uniroma3.it}

\author[M.~Quattropani]{Matteo Quattropani$^{\flat}$}
\address{$^{\flat}$ Dipartimento di Matematica e Fisica, Universit\`a Roma Tre, Largo Murialdo 1, 00146 Roma, Italy.}
\email{matteo.quattropani@uniroma3.it}
%
\date{\today}

\maketitle 
\justifying
\begin{abstract}

We consider the generalised PageRank walk on a digraph $G$, with refresh probability $\alpha$ and resampling distribution $\lambda$.
We analyse convergence to stationarity when $G$ is a large sparse random digraph with given degree sequences, in the limit
of vanishing $\alpha$. We identify three scenarios: when $\alpha$ is much smaller than the inverse of the
mixing time of $G$ the relaxation to equilibrium is dominated by the simple random walk and displays a cutoff
behaviour; when $\alpha$ is much larger than the inverse of the mixing time of $G$ on the contrary one has pure
exponential decay with rate $\alpha$; when $\alpha$ is comparable to the inverse of the mixing time of $G$ there is a mixed
behaviour interpolating between cutoff and exponential decay. This trichotomy is shown to hold uniformly
in the starting point and uniformly in the resampling distribution $\lambda$. 

%

\end{abstract}

\section{Introduction and results}
Given a directed graph $G=(V,E)$ and a parameter $\a\in(0,1)$, 
the PageRank surf on $G$ with damping factor $1-\a$ is the Markov chain 
with state space $V$ and transition probabilities given by
 \begin{equation}\label{de:pstar}
P_\a(x,y)=(1-\a)P(x,y)+\frac\a{n}, 
\end{equation}
where $n=|V|$ is the number of vertices of $G$, and, writing $d_x^+$ for the out-degree of vertex $x$, 
 \begin{equation}\label{de:pstar2}
 P(x,y)=\begin{cases} 1/d_x^+& \text{if } (x,y)\in E\\0 & \text{otherwise}
 \end{cases}
 \end{equation}
 denotes the transition matrix of the simple random walk  on $G$. The interpretation is that of a surfer that at each step, with probability $1-\a$ moves to  a vertex chosen uniformly at random among  the out-neighbours of its current state, and with probability $\a$ moves to a uniformly random vertex in $V$.   
The surfer reaches eventually a stationary distribution $\pi_\a$ over $V$, called the PageRank of $G$.
Since its introduction by Brin and Page in the seminal paper \cite{BrinPage}, PageRank has played a fundamental role in the ranking functions of all major search engines; see e.g.\ \cite{easley2010networks,gleich2015pagerank}. 
A common generalization is the so-called customised or generalised PageRank, where the uniform resampling is replaced by an arbitrary probability distribution $\l$ over $V$, so that \eqref{de:pstar} becomes
\begin{equation}\label{de:pstarla}
P_{\a,\l}(x,y)=(1-\a)P(x,y)+\a\l(y). 
\end{equation}
The resulting stationary distribution $\pi_{\a,\l}$,  characterised by the equation 
\begin{equation}\label{statio}
\pi_{\a,\l}(y) = \sum_{x\in V}\pi_{\a,\l}(x)P_{\a,\l}(x,y),
\end{equation}
 depends in a nontrivial way on the parameter $\a$ and the distribution $\l$. 
 There have been several 
 investigations of the structural properties of $\pi_{\a,\l}$; 
 see e.g.\ 
  ~\cite{jeh2003scaling,andersen2006local,BRESSAN2010199}; we refer in particular to the recent works  \cite{chen2017generalized,GarVDHLit:LWPR2018,vial2019structural} for cases where the graph $G$ is drawn from the configuration model. 
Here we focus on the dynamical problem of determining the time needed for the surfer to reach the equilibrium distribution $\pi_{\a,\l}$, namely we study the mixing time of the Markov chain with transition matrix $P_{\a,\l}$.
In the case $\a=0$, this corresponds to the classical question of determining the mixing time of the simple random walk on the graph $G$; see e.g.\ \cite{LevPer:AMS2017}. Even for graphs where the latter 
is well understood, it is in general  not immediate to deduce the influence of the parameter $\a$ and of the resampling distribution $\l$ on the speed of convergence to equilibrium.  

It is intuitively reasonable to guess that if the parameter $\a$ is suitably large compared to the inverse of the mixing time of the graph $G$, then the time to reach stationarity will be essentially the expected time needed to make the first $\l$-resampling transition, that is a geometric random variable with parameter $\a$, while if $\a$ is suitably small  compared to the inverse of the mixing time of the graph $G$, then one should reach stationarity well before the first $\l$-resampling, so that the speed of convergence to equilibrium will be essentially that of the simple random walk on  $G$. Moreover, one could expect that when $\a$ is neither too small nor too large compared to the inverse of the mixing time of the graph $G$, then some interpolation between the two opposite behaviours should take place. In this paper we substantiate this intuitive picture for a large class of sparse directed graphs. The results hold uniformly in the initial position and uniformly in the resampling distribution $\l$.

 \subsection{Two models of sparse digraphs}\label{sec:twomod}
 We shall consider two families of directed graphs. Both are obtained via the so-called configuration model, with the difference that in the first case we fix both in and out degrees, while in the second case we only fix the out degrees. The models are sparse in that the degrees are bounded. We now proceed with the formal definition.
 
 Let $V$ be a set of $n$ vertices. For simplicity we often write $V=[n]$, with $[n]=\{1,\dots,n\}$.
 For each $n$, we are given two finite sequences $\bd^+=(d_x^+)_{x\in[n]}$ and $\bd^-=(d_x^-)_{x\in[n]}$ of non negative integers such that 
 \begin{equation}\label{degsm}
 m=\sum_{x\in V}d_x^+=\sum_{x\in V}d_x^-.
 \end{equation} 
 The {\em directed configuration model} DCM($\bd^\pm$), is the distribution of the random graph $G$ obtained as follows: 1) equip each node $x$ with $d_x^+$ tails and $d_x^-$ heads; 2) pick uniformly at random one of the $m!$ bijective maps  
 from the set of all tails into the set of all heads, call it $\o$; 3) for all $x,y\in V$, add a directed edge $(x,y)$ every time a tail from $x$ is mapped into a head from $y$ through $\o$. The resulting graph $G$ may have self-loops and multiple edges, however it is classical that by conditioning on the event that there are no multiple edges and no self-loops one obtains a uniformly random simple digraph with in degree sequence $\bd^-$ and out degree sequence $\bd^+$.  
 
 Structural properties of random graphs obtained in this way have been extensively studied in  ~\cite{CooFri:SizeLargSCRandDigr2002}. Here we shall consider the sparse case corresponding to bounded degree sequences. Moreover, in order to avoid non irreducibility issues, we shall assume that all degrees are at least $2$. Thus, throughout this work it will always be assumed that 
 \begin{equation}\label{degs}
\min_{x\in[n]}d_x^-\wedge d_x^+\ge 2,\qquad\max_{x\in[n]}d_x^-\vee d_x^+=O(1). 
\end{equation}
We often use the notation $\D=\max_{x\in[n]}d_x^-\vee d_x^+$. Under the first assumption  it is known that DCM($\bd^\pm$) is strongly connected with high probability; see e.g.~\cite{CooFri:SizeLargSCRandDigr2002}. Under the second assumption, it is known that DCM($\bd^\pm$) has a uniformly (in $n$) positive probability of having no self-loops nor multiple edges; see e.g.\ \cite{chen2013directed}. In particular, any property that holds with high probability for DCM($\bd^\pm$) will also hold with high probability for a uniformly chosen simple digraph subject to the constraint that in and out degrees be given by $\bd^-$ and $\bd^+$ respectively. 
Here and throughout the rest of the paper we say that a property holds {\em with high probability} (w.h.p. for short) if the probability of the corresponding event converges to $1$ as $n\to\infty$. In particular, it follows that w.h.p.\ there exists a unique stationary distribution $\pi_0$ for the simple random walk on $G$. Several properties of $\pi_0$ have been established recently in  \cite{BCS1}, where it was shown, among other facts, that $\pi_0$ can be described in terms of recursive distributional equations determined by the sequences $\bd^\pm$. 
 
 To define the second model, for each $n$ let $\bd^+=(d_x^+)_{x\in[n]}$ be a finite sequence of non negative integers and define the {\em out-configuration model} OCM($\bd^+$) as the distribution of the random graph $G$ obtained as follows: 1) equip each node $x$ with $d_x^+$ tails; 2)  pick, for every $x$ independently, a uniformly random injective map 
from  the set of tails at $x$ to the set of all vertices $V$, call it $\o_x$; 3) for all $x,y\in V$, add a directed edge $(x,y)$ if a tail from $x$ is mapped into $y$ through $\o_x$.
Equivalently, $G$ is the graph whose adjacency matrix is uniformly random in the set of all $n\times n$ matrices with entries $0$ or $1$ such that every row $x$ sums to $d_x^+$. Notice that $G$ may have self-loops, but there are no multiple edges in this construction. This is due to the requirement that the maps $\o_x$ be injective. The latter choice is only a matter of convenience, and everything we say below is actually seen to hold as well for the model obtained by dropping that requirement.   We write $\o=(\o_x)_{x\in [n]}$ for the collection of maps. 
As before we shall make the assumptions  \begin{equation}\label{degs+}
\min_{x\in[n]}d_x^+\ge 2,\qquad\max_{x\in[n]} d_x^+=O(1), 
\end{equation}
and use the notation $\D=\max_{x\in[n]} d_x^+$. 
We remark that under the above assumptions there can still be vertices with in-degree zero, and therefore in this case $G$ is not necessarily strongly connected. However, it is still possible to show that w.h.p.\ there exists a unique 
stationary distribution $\pi_0$ for the simple random walk on $G$; see e.g.\ \cite{addario2015diameter,BCS2} for more details. 

In what follows $G=G(\o)$ denotes a given realization of either the directed configuration model  DCM($\bd^\pm$) or the out-configuration model 
OCM($\bd^+$) and all the results to be discussed will hold w.h.p.\ within these two ensembles. For the sake of simplicity we often refer to these as {\em model 1} and {\em model 2} respectively.

 \subsection{Main results}\label{sec:mainresults}
Let $P$ denote the transition matrix 
of the simple random walk on $G$. When $G$ is a digraph without multiple edges this is given by \eqref{de:pstar2}. If $G$ has multiple edges, $P(x,y)$ is defined as $m(x,y)/d^+_x$ where $m(x,y)$ denotes the number of directed edges from $x$ to $y$.   
For any $\a\in(0,1)$ and any resampling distribution $\l$, let $P_{\a,\l}$ denote the PageRank transition matrix defined in \eqref{de:pstarla}. Notice that as soon as $\a>0$, regardless of the realization of the graph $G$ and of the chosen distribution $\l$, there exists a unique 
stationary distribution $\pi_{\a,\l}$ on $V$. Indeed, the transition matrix  $P_{\a,\l}$ satisfies the so-called Doeblin condition if $\a>0$; see  Proposition \ref{prop:expa} below for an explicit expression of $\pi_{\a,\l}$.  
Convergence to equilibrium will be quantified using the total variation distance. For two probability measures $\mu,\nu$, the latter is defined by
\begin{equation}\label{de:totvar}
\|\mu-\nu\|_\tv = \max_E |\mu(E)-\nu(E)|,
\end{equation}
where the maximum ranges over all possible events in the underlying probability space. 
Starting at a node $x$ the distribution of the PageRank surfer after $t$ steps is  $P^t_{\a,\l}(x,\cdot)$, and the distance to equilibrium is defined by
\begin{equation}\label{dxt}
\cD^x_{\a,\l}(t)=\left\| P^t_{\a,\l}(x,\cdot)-\pi_{\a,\l}\right\|_{\tv}.
\end{equation}
This defines a non-increasing function of $t\in\bbN$. It is convenient to extend it to a monotone function of $t\in[0,\infty)$, e.g.\ by considering the integer part of the argument.  
Finally, for any $\e\in(0,1)$, the 
$\e$-{\em mixing time} is defined by 
\begin{equation}\label{tmixe}
T_{\a,\l}(\e)=\inf \left\{t\geq 0: \;\max_{x\in V}\cD^x_{\a,\l}(t) \leq \e\right\} .
\end{equation}
Both $\cD^x_{\a,\l}(t)$ and $T_{\a,\l}(\e)$ are functions of the underlying graph $G$, and are therefore random variables.  
When $\a=0$, we write $\cD^x_0(t)$ and $T_0(\e)$ for the corresponding quantities. 
The behaviour of the distance $\cD^x_0(t)$ and of the mixing time $T_0(\e)$ has been thoroughly investigated in \cite{BCS1} for model 1 and in \cite{BCS2} for model 2. 
Let us briefly recall the main conclusions of these works. In order to simplify the exposition, we shall adopt the following unified notation. Let us define the {\em in-degree distribution}
\begin{equation}\label{def:muin}
\muin(x)=\frac1n\times \begin{cases}d^-_x /\left\langle d\right\rangle & \text{{\small model 1}}\\
1 & \text{{\small model 2}}
\end{cases}
\end{equation}
where we use the notation $$\left\langle d\right\rangle =\frac1n\sum_{x\in V} d^-_x = \frac{m}n$$ for the average degree. Note that for model 2 the distribution $\muin$ represents the average in-degrees rather than  the actual in-degrees.
Next, let the {\em entropy} $H$ and the associated {\em entropic time} $\tent$ be defined by 
\begin{equation}\label{def:tent}
H=\sum_{x\in V}\muin(x)\log d^+_x, \;\qquad \tent = \frac{\log n}H.
\end{equation}
Note that under our assumptions on $\bd^\pm$ the deterministic quantities $H,\tent$ satisfy $H=\Theta(1)$ and $\tent = \Theta(\log n)$. The main results of \cite{BCS1,BCS2} state that, uniformly in the starting point $x\in V$,  the rescaled function $\cD^x_0(s \,\tent)$, $s>0$, converges in probability as $n\to\infty$ to the step function 
  \begin{equation}\label{def:theta}
\vartheta(s)=\begin{cases}1 & \text{{\small if \,}} s<1\\ 0&\text{{\small if \,}} s>1.
\end{cases}
\end{equation}
More precisely, we may combine \cite[Theorem 1]{BCS1} and \cite[Theorem 1]{BCS2} to obtain the following statement.
\begin{theorem}[Uniform cutoff at the entropic time \cite{BCS1,BCS2}]\label{th:BCS}
Let $G$ be a random graph from either
the directed configuration model  DCM($\bd^\pm$) or the out-configuration model 
OCM($\bd^+$).
For each $s>0, s\neq 1$ one has:  
\begin{equation}\label{cutoff1}
\max_{x\in [n]}\left|\cD^x_0(s \,\tent) - \vartheta(s)
\right|\overset{\P}{\longrightarrow}0.
\end{equation}
\end{theorem}
In \eqref{cutoff1} we use the notation $\overset{\P}{\longrightarrow}$ for convergence in probability as $n\to\infty$. In terms of mixing times,  \eqref{cutoff1} implies in particular that for any $\e\in(0,1)$:
\begin{equation}\label{cutoff2}
\frac{T_0(\e)}{\tent}
\overset{\P}{\longrightarrow}1,
\end{equation}
The fact that the distance to equilibrium approaches a step function, or equivalently that the $\e$-mixing time is to leading order insensitive to the value of $\e\in(0,1)$, is commonly referred to as a {\em cutoff phenomenon}; see e.g.\ \cite{diaconis1996cutoff,LevPer:AMS2017} for a review. We also refer to \cite{lubetzky2010cutoff,ben2017cutoff,berestycki2018random} for similar results in the case of undirected graphs. We stress that a fundamental difference between the case of undirected graphs and the case of directed graphs considered here is that the underlying stationary distribution $\pi_0$ is not known explicitly in the directed case.   

We now formulate our main results. To obtain explicit asymptotic statements we shall assume that $\a=\a(n)\in(0,1)$ is a sequence such that $\a\to 0$ and such that the limit
\begin{equation}\label{def:beta}
\g=\lim_{n\to\infty}\a\, \tent\in[0,\infty]
\end{equation}
exists, with possibly $\g=0$ or $\g=\infty$. 
 We call $\cS_n$ the set of all probability measures on $[n]$.

\begin{theorem}\label{th:general}
Let $G$ be a random graph from either
the directed configuration model  DCM($\bd^\pm$) or the out-configuration model 
OCM($\bd^+$). Let $\a=\a(n)\in(0,1)$ be  parameters as in \eqref{def:beta}.  Then, according to the value of $\g$ there are three scenarios: 
\begin{enumerate}
	\item[$\textbf{(1)}$] If $\g=0$ then for all $s>0, s\neq 1$:
	\begin{equation}\label{scen1}
	\max_{\l\in\cS_n}\max_{x\in[n]}\left|\cD_{\a,\l}^x(s \,\tent)-\vartheta(s)\right|\overset{\P}{\longrightarrow}0.
		\end{equation}
	\item[$\textbf{(2)}$]  If $\g\in\(0,\infty\)$ then for all $s>0,s\neq \g$:
	\begin{equation}\label{scen2}
	\max_{\l\in\cS_n}\max_{x\in[n]}\left|\cD_{\a,\l}^x(s/\a)-e^{-s}\vartheta(s/\g)\right|\overset{\P}{\longrightarrow}0.
		\end{equation}
	\item[$\textbf{(3)}$]  If $\g=\infty$ then for all $s>0$:
	\begin{equation}\label{scen3}
	\max_{\l\in\cS_n}\max_{x\in[n]}\left|\cD_{\a,\l}^x(s /\a)-e^{-s}\right|\overset{\P}{\longrightarrow}0.
		\end{equation}
\end{enumerate}
\end{theorem}

In terms of mixing times, Theorem \ref{th:general} implies the following statements. 
\begin{corollary}\label{cor:general}
In the setting of Theorem \ref{th:general}, the following holds uniformly with respect to $\l$:
\begin{enumerate}
	\item[$\textbf{(1)}$] If $\g=0$ then for all $\e\in(0,1)$
	\begin{equation}\label{2scen1}
\frac{T_{\a,\l}(\e)}{\tent}
\overset{\P}{\longrightarrow}1,
\end{equation}
	\item[$\textbf{(2)}$]  If $\g\in\(0,\infty\)$:
	\begin{equation}\label{2scen2}
	\frac{T_{\a,\l}(\e)}{\tent}
\overset{\P}{\longrightarrow}\begin{cases}
1 & \text{{\small if \,}} \e\in(0,e^{-\g})\\ \frac1\g\log(1/\e)&\text{{\small if \,}} \e\in[e^{-\g},1).
\end{cases}
		\end{equation}
	\item[$\textbf{(3)}$]  If $\g=\infty$ then for all $\e\in(0,1)$:
	\begin{equation}\label{2scen3}
	\a\,T_{\a,\l}(\e)\overset{\P}{\longrightarrow}\log(1/\e).
		\end{equation}
\end{enumerate}
\end{corollary}
The  trichotomy displayed in Theorem \ref{th:general} and Corollary \ref{cor:general} reflects the competition between two distinct mechanisms of relaxation to equilibrium: the simple random walk dominates in the first scenario, while the $\l$-resampling dominates in the third; the intermediate scenario interpolates between the two extremes; see Figure \ref{fig:fig}. 

\begin{figure}[h]
	\centering
	\includegraphics[width=16cm]{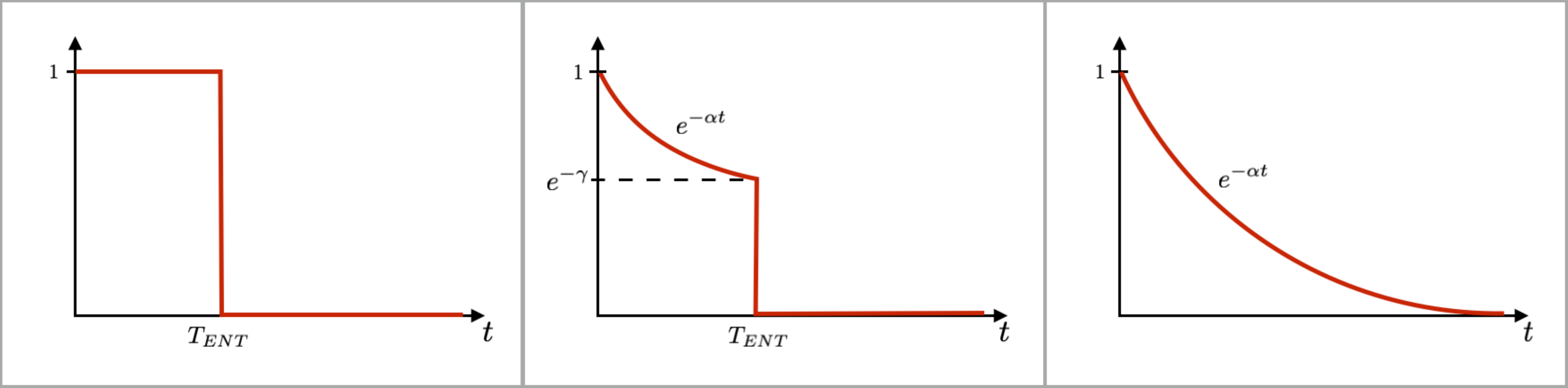}\\
	\caption{}	\label{fig:fig}
\end{figure}

Remarkably, essentially the same  trichotomy was uncovered recently by \cite{AveGulVDHdH:DirConf2018} in a model of random walk on dynamically evolving undirected graphs. In that case, the role of the resampling is played by the underlying reshuffling of the graph edges. It is interesting to observe that, in contrast with the undirected case considered in \cite{AveGulVDHdH:DirConf2018}, in our setting the two competing processes may well have very distinct goals, and the overall stationary distribution $\pi_{\a,\l}$ is the result of a nontrivial balance.  

To give some guidelines, below we illustrate the main ideas involved in the proof. 

The starting point is the observation that the distance to stationarity $\cD_{\a,\l}^x(t)$ satisfies the following general identity at all times $t$, for all choices of the parameter $\a$ and distribution $\l$:
\begin{equation}\label{eq:gen}
\left\| P^t_{\a,\l}(x,\cdot)-\pi_{\a,\l}\right\|_{\tv} = (1-\a)^t\left\| P^t(x,\cdot)-\pi_{\a,\l}P^t\right\|_{\tv}\,.
		\end{equation}
		Here we use the notation $\mu P^t(y)=\sum_{x\in V}\mu(x)P^t(x,y)$ for  the distribution at time $t$ of the simple random walk started at a random vertex distributed according to some distribution $\mu$.
		The relation \eqref{eq:gen} follows from a simple coupling argument; see Proposition \ref{prop:tv} below. Moreover, the stationary distribution admits the power series expansion 
		\begin{equation}\label{eq:gen2}
\pi_{\a,\l}= \a\sum_{k=0}^\infty(1-\a)^k\l P^k\,,
		\end{equation} 
see Proposition \ref{prop:expa} below. 
A particularly simple special case is when the resampling distribution $\l$ equals the stationary distribution $\pi_0$.
Indeed, in this case the stationary distribution is the result of a trivial balance 
and $\pi_{\a,\l}=\pi_0$, so that \eqref{eq:gen} becomes
\begin{equation}\label{spi}
\cD_{\a,\pi_0}^x(t) = (1-\a)^t\cD_{0}^x(t)\,.
		\end{equation}
Therefore, when $\l=\pi_0$ the results in Theorem \ref{th:general} are an immediate consequence of Theorem \ref{th:BCS}. Moreover, this shows that the trichotomy in Theorem \ref{th:general} follows from Theorem \ref{th:BCS}
whenever the distribution $\l\in\cS_n$ is such that 
\begin{equation}\label{eq:wapprox}
\left\| \pi_{\a,\l} - \pi_0\right\|_{\tv}\overset{\P}{\longrightarrow} 0,
\end{equation}
since in this case $\pi_{\a,\l}P^t$ is well approximated by $\pi_0$, and the the three claims in Theorem \ref{th:general} would follow from \eqref{eq:gen}.
As we shall see, the approximation \eqref{eq:wapprox} is rather straightforward in the first scenario. Indeed, if $\a\tent\to 0$ then the simple random walk has enough time to reach equilibrium between successive resampling events and \eqref{eq:wapprox} holds uniformly in $\l\in\cS_n$, see Proposition \ref{prop:ex1} below.
The second and third scenarios require a different approach since one cannot expect \eqref{eq:wapprox} to hold for all $\l\in\cS_n$. There is however a special class of distributions, that we refer to as {\em widespread}, which does satisfy \eqref{eq:wapprox}  in all three scenarios.

\begin{definition}[Widespread measure]\label{de:widespread}
	A sequence of probability measures $\lambda= \lambda_n$ on $[n]$ is 
	\emph{widespread} if 
	\begin{enumerate}
		\item[$\textbf{(i)}$]  
		There exists $\d>0$ such that  
		\begin{equation}\label{def:ws1}
		|\l|_\infty =\max_{x\in[n]}\lambda(x)=O(n^{-1/2-\d}).
		\end{equation}
		\item[$\textbf{(ii)}$] Bounded $\ell_2$-distance from the uniform distribution:
		\begin{equation}\label{def:ws2}
		\frac{1}{n}\sum_{j\in[n]}\(1-n\lambda(j) \)^2=O(1).
		\end{equation}
	\end{enumerate}
\end{definition}
Note that there is no requirement on the minimum of $\l(x)$, so that large portions of the set of vertices are allowed to receive zero mass. 
An important property of widespread measures is that, if we start with such a distribution $\l$, then the time needed to reach  stationarity for the simple random walk is much smaller than the entropic time $\tent$. More precisely we shall establish the following facts. 
 \begin{lemma}\label{le:approx}
Let $G$ be a random graph from either
the directed configuration model  DCM($\bd^\pm$) or the out-configuration model 
OCM($\bd^+$). If $\l=\l_n$ is widespread, then for any sequence $t=t(n)\to\infty$, 
\begin{equation}\label{eq:approx}
\left\|\l P^t - \pi_0\right\|_{\tv}\overset{\P}{\longrightarrow} 0.
		\end{equation}
		Moreover, in all three scenarios \eqref{eq:wapprox} holds for every widespread distribution $\l$.
\end{lemma}
The result in Lemma \ref{le:approx} illustrates well the mechanism behind the trichotomy in the case of widespread measures $\l$, but it is far from explaining the general phenomenon described in Theorem \ref{th:general}.
For instance, if $\l=\d_z$ is a Dirac mass at a vertex $z$, then $\l P^t=P^t(z,\cdot)$ and therefore \eqref{eq:approx} must fail for all $t=s\tent$, with $s\in(0,1)$ fixed, since by Theorem \ref{th:BCS} we know that in this case
\begin{equation}\label{eq:napprox}
\min_{z\in[n]}\left\| P^t(z,\cdot) - \pi_0\right\|_{\tv}\overset{\P}{\longrightarrow} 1.
		\end{equation}
Moreover, the stationary distribution $\pi_{\a,\d_z}$ can be very far from $\pi_0$ in both scenarios 2 and 3. In particular, using our analysis in Section \ref{sec:loc} one can check that in scenario 3,  
\begin{equation}\label{eq:napproxo}
\min_{z\in[n]}\left\| \pi_{\a,\d_z}-\pi_0\right\|_{\tv}\overset{\P}{\longrightarrow} 1.
		\end{equation}
While we believe the result in Lemma \ref{le:approx} to be of interest in its own, the proof of Theorem \ref{th:general} will be based on a different approach. 

The first observation is that the identity \eqref{eq:gen} together with the result of Theorem \ref{th:BCS} is already sufficient to establish all the upper bounds on the distance $\cD_{\a,\l}^x(t)$ required in the proof of Theorem \ref{th:general}, see Section \ref{sec:trichotomy} for the details. On the other hand, some extra work is needed for the proof of the lower bounds on $\cD_{\a,\l}^x(t)$. 
 A key technical point for establishing the desired lower bounds will be the following fact concerning  scenarios 2 and 3.
\begin{lemma}\label{le:singular}
Let $G$ be a random graph from either
the directed configuration model  DCM($\bd^\pm$) or the out-configuration model 
OCM($\bd^+$). 
For fixed $\g>0$, including $\g=\infty$, and $s\in(0,\g)$, for any sequence $\a\to 0$, satisfying $\a\tent\to \g$, and $t=s/\a$:  
\begin{equation}\label{eq:singular}
\min_{\l\in\cS_n}\min_{x\in [n]} \left\| P^t(x,\cdot) - \pi_{\a,\l}P^t\right\|_{\tv}\overset{\P}{\longrightarrow} 1\,.
\end{equation}
\end{lemma}
Essentially, \eqref{eq:singular} says that the $t$-step evolution of the random walk starting at any given vertex $x$ is singular with respect to the evolution starting at the page rank distribution, as soon as $t\leq (1-\e)T_\ent$ for some fixed $\e>0$. The uniformity in $x\in [n]$ and $\l\in\cS_n$ in Lemma \ref{le:singular} is a delicate matter. We shall see that for general $\l\in\cS_n$, if $\a \tent \to \g>0$, then $\pi_{\a,\l}P^t$ is a nontrivial mixture of $\pi_0$ and another measure $\mu_\l$, see Lemma \ref{le:decomp} below for the precise version of this statement. 
Depending on the nature of $\l\in\cS_n$, the measure $\mu_\l$ can be either supported on a small subset of $[n]$, e.g.\ if $\l=\d_z$ for some $z$, or very spread out, e.g.\ if $\l$ is widespread as in Definition \ref{de:widespread}. We shall however show that structural features of the random random graph $G$ and the fact that $\a\to0$  imply that the measure $\mu_\l$ cannot concentrate any mass on the support of the distribution $P^t(x,\cdot)$ and thus $\mu_\l$ and $P^t(x,\cdot)$ are approximately singular.  We refer to Section \ref{le:singular} for the derivation of this anti-concentration phenomenon. 
Since $\pi_0$ and $P^t(x,\cdot)$ are approximately singular for $t \leq (1-\e)T_\ent$ as in \eqref{eq:napprox}, this will be sufficient to prove Lemma \ref{le:singular}.

%

The rest of the paper is arranged as follows: the next section establishes the basic identities \eqref{eq:gen} and \eqref{eq:gen2} and some more preliminary material; Section \ref{sec:loc} contains our main technical estimates and the proof of Lemma \ref{le:singular}; Section \ref{sec:trichotomy} shows how to derive the main results from Lemma \ref{le:singular} and the facts established in Section \ref{sec:prelim}. The discussion of widespread measures and the proof  of Lemma \ref{le:approx} form an independent piece of work and are given in Section \ref{sec:martingales}.
 
 \section{Preliminaries}\label{sec:prelim}
Here we collect some simple general facts about the PageRank surf. The statements in this section do not depend on the graph $G$ where the original walk takes place.   Therefore, we fix an arbitrary  digraph $G$ with vertex set $V=[n]$, and let $P$ be the transition matrix in \eqref{de:pstar2}. If $d^+_x=0$ for some $x$ we may define $P(x,x)=1$ and $P(x,y)=0$ for all $y\in V\setminus\{x\}$.   
\subsection{The stationary distribution $\pi_{\a,\l}$}

\begin{proposition}\label{prop:expa}
For any $\a\in(0,1)$, any probability vector $\l$, let $P_{\a,\l}$ be defined by \eqref{de:pstarla}. There exists a unique probability vector $\pi_{\a,\l}$ satisfying $\pi_{\a,\l}P_{\a,\l}=\pi_{\a,\l}$. Moreover, $\pi_{\a,\l}$ is given by 
\begin{equation}\label{eq:expa}
\pi_{\a,\l}=\a\sum_{k=0}^\infty(1-\a)^k\l P^k	.	
\end{equation}
\end{proposition}
\begin{proof}
The equation $\pi_{\a,\l}P_{\a,\l}=\pi_{\a,\l}$ is equivalent to 
$$
\pi_{\a,\l}({\bf 1}-(1-\a)P)=\a\l.
$$
Since $P$ is a stochastic matrix, the matrix ${\bf 1}-(1-\a)P$ is strictly diagonally dominant, and therefore invertible. Then  \eqref{eq:expa} follows by expanding the expression $ \pi_{\a,\l}=\a\l({\bf 1}-(1-\a)P)^{-1}$.
\end{proof}
In particular, \eqref{eq:expa} and the triangle inequality imply
that for any other probability vector $\mu$:
\begin{equation}\label{eq:expa2}
\|\pi_{\a,\l}-\mu\|_\tv\leq\a\sum_{k=0}^\infty(1-\a)^k\|\l P^k-\mu\|_\tv	.	
\end{equation}

 \subsection{Walk vs. teleport}
 A trajectory of the PageRank surf can be sampled as follows. At each time unit independently, we flip a $\a$-biased coin: if heads (with probability $\a$) then the surfer is teleported to a new vertex, chosen according to $\lambda$; if tails (with probability $1-\a$) then the surfer walks one step according to the transition matrix $P$. The probability associated to this construction will be denoted by $\bbP$. If $\t_\a$ denotes the first time the surfer is teleported, then for all $t\in\bbN$: 
 	\begin{equation}\label{eq:easyub1}
\P(\t_\a>t)=(1-\a)^t.
	\end{equation}
	 \begin{proposition}\label{prop:tv}
 For any $\a\in(0,1)$, any probability vector $\l$, and all $t\in\N$, $x\in[n]$:
 	\begin{equation}\label{eq:ube1}
\|P_{\a,\l}^t(x,\cdot)-\pi_{\a,\l}\|_\tv=(1-\a)^t\|P^t(x,\cdot)-\pi_{\a,\l}P^t\|_\tv.
	\end{equation}
\end{proposition}
\begin{proof}
We use the construction introduced above, and write $X^x_t$ for the position of the surfer at time $t$ with initial vertex $x$. By using the same sample of the teleporting distribution $\l$ we couple two trajectories $X^x_t,X^z_t$ in such a way that $ X^x_t=X^z_t$, for all $t\geq \t_\a$.
Therefore, letting $\bbE$ denote the expectation with respect to this coupling:
\begin{align}
P_{\a,\l}^t(x,y)-P_{\a,\l}^t(z,y)&=\bbE\left[\ind(X^x_t=y)-\ind(X^z_t=y)
\right]\nonumber\\
& =
\bbE\left[\ind(X^x_t=y)-\ind(X^z_t=y)
; \t_\a> t\right].
	\end{align}
Moreover,
	\begin{align}\label{eq:la=pi3}
\bbE\left[\ind(X^x_t=y)
; \t_\a> t\right]=\bbP(\t_\a> t)\bbP(X_t=y|X_0=x,\,\t_\a> t) = \bbP(\t_\a> t)P^t(x,y).
	\end{align}
Therefore,
	\begin{align}\label{eq:la=pi4}
P_{\a,\l}^t(x,y)-P_{\a,\l}^t(z,y)&=
\bbP(\t_\a> t)(P^t(x,y)-P^t(z,y)).
	\end{align}
	Multiplying by $\pi_{\a,\l}(z)$, summing over $z$, and using \eqref{eq:easyub1} one obtains
	 	\begin{align}\label{eq:la=pi5}
P_{\a,\l}^t(x,y)-\pi_{\a,\l}(y)&=
(1-\a)^t\left(P^t(x,y)-[\pi_{\a,\l}P^t](y)\right).
	\end{align}
It follows that
	 	\begin{align}\label{eq:la=pi6}
\|P_{\a,\l}^t(x,\cdot)-\pi_{\a,\l}\|_\tv&=\frac12\sum_{y\in V}|P_{\a,\l}^t(x,y)-\pi_{\a,\l}(y)
|\nonumber\\
&=(1-\a)^t\frac12\sum_{y\in V}\left|P^t(x,y)-[\pi_{\a,\l}P^t](y)\right|\nonumber\\
&=(1-\a)^t\|P^t(x,\cdot)-\pi_{\a,\l}P^t\|_\tv.
	\end{align}
\end{proof}
Since the total variation distance is always bounded above by $1$,  Proposition \ref{prop:tv} implies the upper bound 
	\begin{align}\label{eq:ubo}
\cD^x_{\a,\l}(t)=\|P_{\a,\l}^t(x,\cdot)-\pi_{\a,\l}\|_\tv\leq (1-\a)^t.
	\end{align}
The latter, in turn, gives the 
following upper bound on the mixing time. 
 \begin{corollary}\label{cor:easyub}
 For any $\a\in(0,1)$, any probability vector $\l$, 
 and all $\e\in(0,1)$, the $\e$-mixing time \eqref{tmixe} satisfies
		\begin{equation}\label{eq:easyub}
T_{\a,\l}(\e)\leq \frac1\a\log(1/\e).	\end{equation}
\end{corollary}
A further immediate consequence of Proposition \ref{prop:tv}  is that 
if $\l$ is stationary for $P$, then the distance to equilibrium $\cD^x_{\a,\l}(t)$ takes a simple form. 
\begin{corollary}\label{cor:la=pi}
 For any $\a\in(0,1)$, for all $x\in V$ and all $t\in\N$, 
 if $\pi_0$ is a probability vector such that $\pi_0P=\pi_0$, then taking $\l=\pi_0$,
 	\begin{equation}\label{eq:la=pi1}
\cD^x_{\a,\pi_0}(t)=(1-\a)^t\,\|P^t(x,\cdot)-\pi_0\|_\tv.
	\end{equation}
\end{corollary}
\begin{proof}
From Proposition \ref{prop:expa} it follows that $\pi_{\a,\pi_0}=\pi_0$, and therefore $\pi_{\a,\pi_0}P^t=\pi_0$ for all $t$. \end{proof}

Finally, another useful consequence of Proposition \ref{prop:tv} is that it allows us to control the distance $\cD^x_{\a,\l}(t)$ in terms of the distance $\cD^x_{\a,\pi_0}(t)$, for some stationary $\pi_0$ as in Corollary \ref{cor:la=pi}, by means of the distance between $\pi_{\a,\l}$ and $\pi_0$.

\begin{corollary}\label{cor:comp}
 For any $\a\in(0,1)$, all $t\in\N$, any probability vector $\l$,  
 if $\pi_0$ is such that $\pi_0P=\pi_0$,
 	\begin{equation}\label{eq:comp1}
\max_{x\in V}\left|\cD^x_{\a,\l}(t)-\cD^x_{\a,\pi_0}(t)\right|\leq
\|\pi_{\a,\l}-\pi_0\|_\tv.
	\end{equation}
\end{corollary}
\begin{proof}
From the triangle inequality and the fact that $\|\mu P^t-\nu P^t\|_\tv$ is monotone in $t$ for all distributions $\mu,\nu$, one has
 	\begin{equation}\label{eq:comp2}
\left|\|P^t(x,\cdot)-\pi_{\a,\l}P^t\|_\tv - \|P^t(x,\cdot)-\pi_0\|_\tv\right| \leq  
\|\pi_{\a,\l}-\pi_0\|_\tv.
	\end{equation}
The conclusion then follows from 
Proposition \ref{prop:tv} and Corollary \ref{cor:la=pi}. \end{proof}

  \section{Main technical estimates}\label{sec:loc}
  The goal of this section is to prove Lemma \ref{le:singular}.  The proof is divided into three main steps. The first step is a  decomposition of $\pi_{\a,\l}P^t$ as a mixture of $\pi_0$ and a distribution $\mu_\l$ defined below. The second and most delicate step is the proof that $\mu_\l$ and $P^t(x,\cdot)$ are approximately singular for $t$ and $\a$ as in Lemma \ref{le:singular}. The third step concludes the desired result collecting the technical estimates established in the first two steps.  
  \subsection{Decomposition of $\pi_{\a,\l}P^t$}
  We start with a useful decomposition of $\pi_{\a,\l}P^t$ as a mixture of $\pi_0$ and a distribution $\mu_\l$ defined as follows.
  Fix $\eta\in(0,1/2)$,  $t\leq (1-2\eta)\tent$, and define $\mu_\l=\mu^{\eta,t}_\l$ and $A=A^{\eta,t}$ as
   \begin{align}
\label{eq:decomp2}
A=\sum_{k=0}^{(1-\eta)\tent - t}\a(1-\a)^k\,,\qquad \mu_\l =\frac1A \sum_{k=0}^{(1-\eta)\tent - t}\a(1-\a)^k\lambda P^{k+t}.
\end{align} 
Note that $\mu_\l$ depends on the graph $G$ while $A$ is deterministic. 
We consider the case $\a\tent\to \g\in(0,\infty]$ and treat the two cases $\g=\infty$ and $\g\in(0,\infty)$ separately. 
 \begin{lemma}\label{le:decomp0}
Fix
$s\in(0,\infty)$ and assume $\a\tent\to +\infty$ and $t=s/\a$. For all $\e>0$, there exists $\eta>0$ such that with high probability:
\begin{align}
\label{eq:decompo1}
\max_{\l\in\cS_n}\|\pi_{\a,\l}P^t - \mu_\l \|_\tv\leq \e,
\end{align} 
and the normalization in \eqref{eq:decomp2} satisfies $A\geq 1-\e$. 
\end{lemma}
\begin{proof}
Since $\a\tent\to +\infty$ we have $t\ll \tent$. It follows that $A\to 1$. Using Proposition \ref{prop:expa}, 
\begin{align}
\label{eq:}
\pi_{\a,\l}P^t - A\mu_\l = \sum_{k=(1-\eta)\tent - t}^{\infty}\a(1-\a)^k\l P^{k+t},
\end{align} 
and therefore $\|\pi_{\a,\l}P^t - \mu_\l \|_\tv\leq (1-A)$. \end{proof}

\begin{lemma}\label{le:decomp}
Fix  $\g\in(0,\infty)$,
$s\in(0,\g)$ and assume $\a\tent\to \g$ and $t=s/\a$. For all $\e>0$, there exists $\eta>0$ such that with high probability:
\begin{align}
\label{eq:decomp1}
\max_{\l\in\cS_n}\|\pi_{\a,\l}P^t - A\mu_\l - (1-A)\pi_0\|_\tv\leq \e,
\end{align} 
where $A=A^{\eta,t}$ and $\mu_\l=\mu^{\eta,t}_\l$ are given in \eqref{eq:decomp2}.
\end{lemma}
\begin{proof}
For any $a<b$, $z\in[n]$, define the probability vector
\begin{align}
\label{eq:decomp3}
\nu^z_{a,b}=\frac1{Z_{a,b}}
\sum_{k=a\tent}^{b\tent -1}\a(1-\a)^kP^{k+t}(z,\cdot)\,,\qquad 
Z_{a,b}=\sum_{k=a\tent}^{b\tent -1}\a(1-\a)^k.
\end{align} 
Since $t=s/\a$, $s\in(0,\g)$, and $\a\tent\to \g$ we may take $\eta>0$ small enough and assume that $t=\k\tent$, $\k\in(0,1-2\eta)$. 
Using Proposition \ref{prop:expa}, letting $\d_z$ denote the Dirac mass at $z$:
\begin{align}
\label{eq:decompa2}
\pi_{\a,\d_z}P^t = Z_{0,1-\eta-\k} \,\nu^z_{0,1-\eta-\k}+ Z_{1-\eta-\k,1+\eta-\k} \,\nu^z_{1-\eta-\k,1+\eta-\k} + Z_{1+\eta-\k,\infty} \,\nu^z_{1+\eta-\k,\infty}
\end{align} 
Note that $Z_{0,1-\eta-\k} =A$, and $\nu^z_{0,1-\eta-\k}=\mu_{\d_z}$. We show that the middle term above is negligible and that $\nu^z_{1+\eta-\k,\infty}$ is well approximated by $\pi_0$. If $\a\tent\to\g\in(0,\infty)$, by Riemann integration it follows that for all $n$ large enough
\begin{align}
\label{eq:decomp4}
Z_{1-\eta-\k,1+\eta-\k}\leq \sum_{k=(1-\eta-\k)\tent}^{(1+\eta-\k)\tent }\a e^{-k\a}\leq C\eta,
\end{align} 
for some constant $C>0$. 
Next, using the monotonicity in time  of total variation distance and Theorem \ref{th:BCS}, w.h.p.\ 
\begin{align}
\label{eq:decomp5}
\max_{z\in[n]}\sup_{k\geq (1+\eta-\k)\tent} \|P^{k+t}(z,\cdot)-\pi_0\|_\tv \leq
\max_{z\in[n]} \|P^{(1+\eta)\tent}(z,\cdot)-\pi_0\|_\tv 
\leq \eta.
\end{align} 
It follows that w.h.p.\
\begin{align}
\label{eq:decomp6}
\max_z\|\nu^z_{1+\eta-\k,\infty}-\pi_0\|_\tv\leq \eta. 
\end{align} 
Writing $\pi_{\a,\l}P^t =\sum_z \l(z)\pi_{\a,\d_z}P^t$ and taking $A$  and $\mu_\l$ as in  \eqref{eq:decomp2} 
concludes the proof.
\end{proof}

\subsection{Singularity of $\mu_\l$ and $P^t(x,\cdot)$}\label{sec:anticon}
The key to this result is a property of the random walk that was established in \cite{BCS1,BCS2}. Roughly speaking this says  that with high probability, for most vertices $x$, the trajectory $\{X_u^x, \,u\leq t\}$ of the walk started at $x$ up to time $t$ is supported by a ``small" directed tree $\cT_x(t)$ rooted at $x$ provided that $t\leq (1-\eta)\tent$ where $\eta$ is an arbitrary positive constant.  As a result the distribution $P^t(x,\cdot)$ is rather strongly localized. We shall see that the distribution $\mu_\l$, depending on the nature of $\l$, could be either supported on a small subset of $[n]$ (e.g.\ if $\l=\d_z$ for some $z\in[n]$) or very spread out (e.g.\ if $\l$ is widespread). The approximate singularity of $\mu_\l$ and $P^t(x,\cdot)$ turns out to be the result of a delicate structural property of the digraph $G$ which guarantees that even if $\mu_\l$ is localized it must be sufficiently smeared out and cannot concentrate on the support of $P^t(x,\cdot)$. 
We first recall the construction of the tree $\cT_x(t)$ and then address the structural properties ensuring this anti-concentration. 

\subsubsection{The tree $\cT_z(t)$}\label{sec:tree}
Given the digraph $G$, the tree $\cT_z(t)$, for fixed $t\leq (1-\eta)\tent$,  can be discovered  algorithmically as described in 
 \cite[Section 6.2]{BCS1} and \cite[Section 4.1]{BCS2}.
We recall the detailed construction for model 1. A very similar construction can be given for model 2; see \cite[Section 4.1]{BCS2}.

Below we describe a sequence of digraphs $\cG^0,\cG^1,\dots,\cG^\k$ such that at each step $\cG^\ell$ is a subset of the out-neighborhood of $z$ of height $t$ in $G$ and such that $\cG^\ell$ is obtained from $\cG^{\ell-1}$
by adding a single edge of $G$. Moreover,  we obtain a sequence of directed trees $\cT^0,\cT^1,\dots,\cT^\k$ such that for every $\ell$, $ \cT^\ell$ is a spanning tree of $\cG^\ell$. The tree $\cT_z(t)$ will be defined as $\cT_z(t)=\cT^\k$.

Initially all matchings of tails and heads in $G$ are unrevealed
and $\cG^0=\cT^0=\{z\}$; let $\partial_- \cT^\ell$ (resp. $\partial_+\cT^\ell$) denote the set of unrevealed heads (resp. tails)
whose endpoint belongs to $\cT^\ell$; the height $h(e)$ of a tail $e\in\partial_+\cT^\ell$ is defined as $1$ plus the number of edges in the unique path in $\cT^\ell$ from $z$ to the endpoint of $e$; the weight of $e\in\partial_+\cT^\ell$ is defined as 
\begin{equation}\label{eq:weighte}
\w(e) = \prod_{i=0}^{h(e)-1}\frac1{d_{x_i}^+}\,,
\end{equation}
where $(z=x_0,x_1,\dots,x_{h(e)-1})$ denotes the path in $\cT^\ell$ from $z$ to the endpoint of $e$;
we then iterate the following steps:
	\begin{itemize}
		\item a tail $e\in \partial_+\cT^\ell$ is selected with maximal weight among all $e\in\partial_+\cT^\ell$ with 
		 $h(e) \leq t$ and 
		  $\w(e) \geq \w_{\min}:=n^{-1+\eta^2}$ (using an arbitrary ordering of the tails to break ties); 
		
		\item the head $f$ matched to $e$ in $G$ is revealed, and $\cG^{\ell+1}$ is obtained from $\cG^\ell$ by adding the edge $ef$;
		
		\item if $f$ was not in $\partial_-\cT^\ell$, then its endpoint and the edge $ef$ are added to  $\cT^\ell$ to form $\cT^{\ell+1}$.
	\end{itemize}
	The process stops at $\ell=\k$ when there are no tails $e\in\partial_+\cT^\k$ with height $h(e) \leq t$ and weight $\w(e)\geq \w_{\min}$.  
Note that $\cT^\ell$ is a directed spanning tree of $\cG^\ell$ at each step. The tree $\cT_z(t)$ is defined as $\cT^\k$. 
After the construction of the tree $\cT_z(t)$, exactly $\kappa$  edges of $G$ have been revealed, some of which may not belong to $\cT_z(t)$. Note that $\cG^\k$ has $\k$ edges and coincides with the union of all directed paths from $z$ which have length at most $t$ and at least probability $\w_{\min}$ with respect to the random walk started at $z$.
As in \cite[Lemma 11]{BCS1}, \cite[Lemma 7]{BCS2}, it is not difficult to see that when exploring the out-neighborhood of $z$ in this way the number $\kappa$, regardless of the realization of $G$, is bounded as 
	\begin{equation}\label{eq:kappaw}
	\kappa\leq n^{1-\frac{\eta^2}2}.
	\end{equation}
	Let us recall the following key facts established in \cite[Section 6]{BCS1} for model 1 and in \cite[Section 4]{BCS2} for model 2.
For every $\eta>0$, for every $t\leq (1-\eta)\tent$, the trajectory $(X_0,\dots,X_t)$ of the random walk  
started at $z$ in $G$ satisfies with high probability $(X_0,\dots,X_t)\subset\cT_z(t)$ for most initial positions $z$. More precisely, let $Q_z(\cdot)$ denote the quenched law of the random walk $(X_0,X_1\dots)$ in $G$ started at $z$.  
Let $V_*$ denote the set of $z\in [n]$ such that $B^+_{z,\hslash}$
 is a directed tree, where $\hslash:=\frac1{10}\log_\Delta(n)$, and $B^+_{z,\hslash}$ denotes the out neighborhood of $z$ of height $\hslash$ in $G$ (that is the subgraph of $G$ induced by the set of vertices which can be reached from $z$ with a path of length at most $\hslash$).
Then,  from  \cite[Proposition 10, part (ii)]{BCS1}, and  
\cite[Lemma 11]{BCS2}, one has
\begin{align}\label{eq:tauze2}
\min_{z\in V_*}Q_z\left((X_0,\dots,X_t)\subset \cT_z(t)\right)\overset{\P}{\longrightarrow}1,
\end{align}
where the notation $(X_0,\dots,X_t)\subset \cT_z(t)$ indicates that the walk up to time $t$ traverses only edges of $ \cT_z(t)$.

\subsubsection{Key technical estimate}
Let $\cA_{x}(t)$ denote the set of vertices in $\cT_x(t)$ that have distance from $x$ exactly $t$ in $\cT_x(t)$. Recall the definition $\mu_\l=\mu_\l^{\eta,t}$ in \eqref{eq:decomp2}. 
\begin{lemma}\label{le:key}
Assume $\a\tent\to \g\in(0,\infty]$. Fix $\eta\in(0,1/2)$ and take $t\leq (1-2\eta)\tent$. Then, for all $\e>0$,  with high probability  
\begin{equation}
\label{eq:key1}
\max_{\l\in\cS_n}\max_{x\in[n]}\mu_\l(\cA_x(t))\leq \e.
\end{equation} 
\end{lemma}
The proof of Lemma \ref{le:key} is based on a structural property of the graph $G$ which says that the intersections of the trees $\cT_z(u)$ and $\cT_x(t)$, where $x,z$ are two arbitrary vertices and $t\leq u\leq (1-\eta)\tent$, are such that, with high probability, for all $x,z\in[n]$, no path in $\cT_z(u)$ can intersect more than $K$ times the set $\cA_x(t)$ where $K$ is a suitably large constant. 
Let us use the notation $\cP_{z}(u)$ for the set of paths in $\cT_z(u)$ having length exactly $u$ and, for all $\p\in \cP_{z}(u)$, let $V(\p)$ denote the set of vertices along that path. Note that the endpoint of $\p\in \cP_{z}(u)$ is necessarily a vertex of $\cA_z(u)$ and $|V(\p)|=u$, since $\cT_z(u)$ is a tree.
	\begin{lemma}\label{lemma:new}
	Fix $\eta\in(0,1/2)$. For every $x,z\in[n]$  and $t\leq u\le (1-\eta)T_\ent$,
	\begin{equation}\label{eq:new1}
	\P\(\exists\: \p\in\cP_z(u):\: \cA_x(t)\cap V(\p)> K    \) \le n^{-3},
	\end{equation}
	 for all $n$ large enough, where
$K=(9+3\log_2\Delta)/\eta^2$. In particular, the event 
$$
\cE_{K} = \{\forall x,z\in[n],\, \forall\p\in\cP_z(u),\, \cA_x(t)\cap V(\p)\le  K\}
$$
holds with high probability.
	\end{lemma}
	
\begin{figure}[h]
	\centering
	\includegraphics[width=6cm]{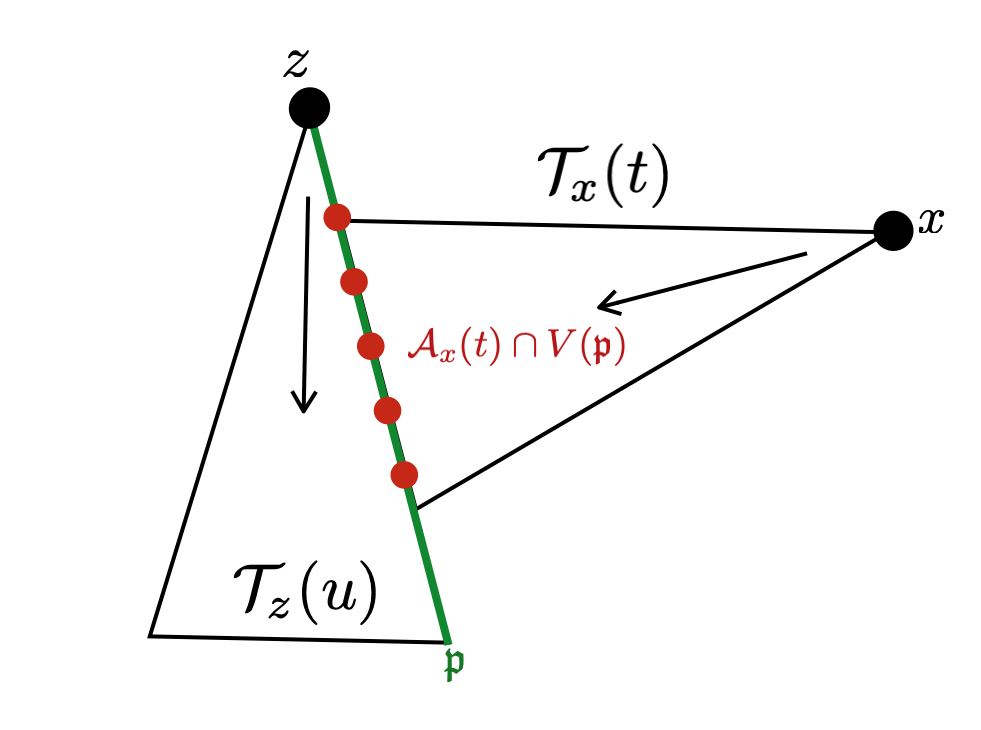}\\
	\caption{A sketch of the unlikely event in Lemma \ref{lemma:new}}	\label{fig:figbad1}
\end{figure}

	\begin{proof}
	We sample the pair $(\cT_x(t),\cT_z(u))$ in the random digraph $G$ by generating first the subgraph $\cT_x(t)$, and then the subgraph $\cT_z(u)$ conditionally on $\cT_x(t)$. The construction of $\cT_x(t)$ follows the steps described by the algorithm in Section \ref{sec:tree} with the understanding  that,  for model 1 the head $f$ to be matched to the tail $e$ to form  $\cG^{\ell+1}$ is chosen uniformly at random among all $m-\ell$ heads that are unmatched after the $\ell$-th step, while for model 2 the tail $e$  has to be connected to a uniformly random vertex in $[n]$.  The process terminates when the tree $\cT_x(t)$ has been fully generated after $\k_x$ steps. A crucial feature of this construction is that the tails of all vertices $v\in\cA_x(t)$ are unmatched once the tree $\cT_x(t)$ has been generated. 
	Moreover, the number of vertices of $\cT_x(t)$ satisfies, as in \eqref{eq:kappaw}
		\begin{equation}\label{eq:kappaw2}
|\cT_x(t)|\leq 	\kappa_x\leq n^{1-\frac{\eta^2}2}.
	\end{equation}
		Next, we generate the tree $\cT_z(u)$, conditionally on $\cT_x(t)$. This is done by starting at $z$ and by repeating the same steps for the construction described in Section \ref{sec:tree} with the difference that if at step $\ell$ a tail $e$ is chosen which had already been matched during the generation of $\cT_x(t)$ then the corresponding edge is included in the construction (and possibly in the tree being generated). The process terminates when the tree $\cT_z(u)$ has been fully generated after $\k_z$ steps. Thus, after $\k_x+\k_z$ steps we have a sample from the joint distribution of $\cT_x(t)$ and $\cT_z(u)$ in $G$. Note that the total number of edges of $G$ discovered after the generation of both trees is 
$\kappa_x+\kappa_z\leq  2n^{1-\eta^2/2}$.
	Let $\{\cF_\ell\}$ denote the filtration associated to this  generation process, so that $\cF_{\k_x}$ is the $\si$-field associated to the tree $\cT_x(t)$.  
	
	During the process generating $\cT_z(u)$ conditionally on $\cT_x(t)$, we say that a {\em bad matching} occurs at step $\ell$ if the tail chosen at that step is currently unmatched (that is it was not revealed during the sampling of $\cT_x(t)$) and it gets connected to a vertex $y$ that was already discovered  in $\cT_x(t)$.      The first key observation is that the conditional probability of a bad matching at step $\ell$ given $\cF_\ell$ is uniformly bounded above by 
	\begin{equation}\label{eq:pp}
	p:=2\D n^{-\eta^2/2}.
	\end{equation}	
	Indeed, in the case of model 1 this probability is at most $\k_x\D/ (m-\k_x-\k_z)$, while for model 2 this probability is at most $\k_x/ n$. In either case it is less than the number $p$ defined in \eqref{eq:pp} for all $n$ large enough. 
	
	The second key observation is that if a path $\p\in\cP_z(u)$ is such that  $\cA_x(t)\cap V(\p)> K$ then at least $K$ bad matchings have occurred during the formation of that given path. To see this, observe that after a vertex $y\in\cA_x(t)$ is visited for the first time during the construction of $\cT_z(u)$, the tails of $y$ will be all matched (at suitable steps $\ell_{1},\dots,\ell_{d_y^+}$) to a uniformly sampled head among the ones that are currently unmatched (for model 1) or to a uniformly random vertex (for model 2). Indeed, the tails of all vertices $v\in\cA_x(t)$ all have the same weight $\w(e)$ and all of them are unmatched after the tree $\cT_x(t)$ has been generated. 
	Also, by definition, every path in $\cT_z(u)$ can visit a given vertex $y$ at most once, and after a visit to $y\in\cA_x(t)$ it has to return to $\cT_x(t)$ with a bad matching in order to visit some other $y'\in\cA_x(t)$. Hence, the number of visits to $\cA_x(t)$ in a given path $\p\in\cP_z(u)$ is at most the number of bad matchings occurred along that path +1. The extra 1 comes from the fact that $z$ could have started already inside $\cT_x(t)$, for instance if $z\in\cA_x(t)$.

\begin{figure}[h]
	\centering
	\includegraphics[width=4.5cm]{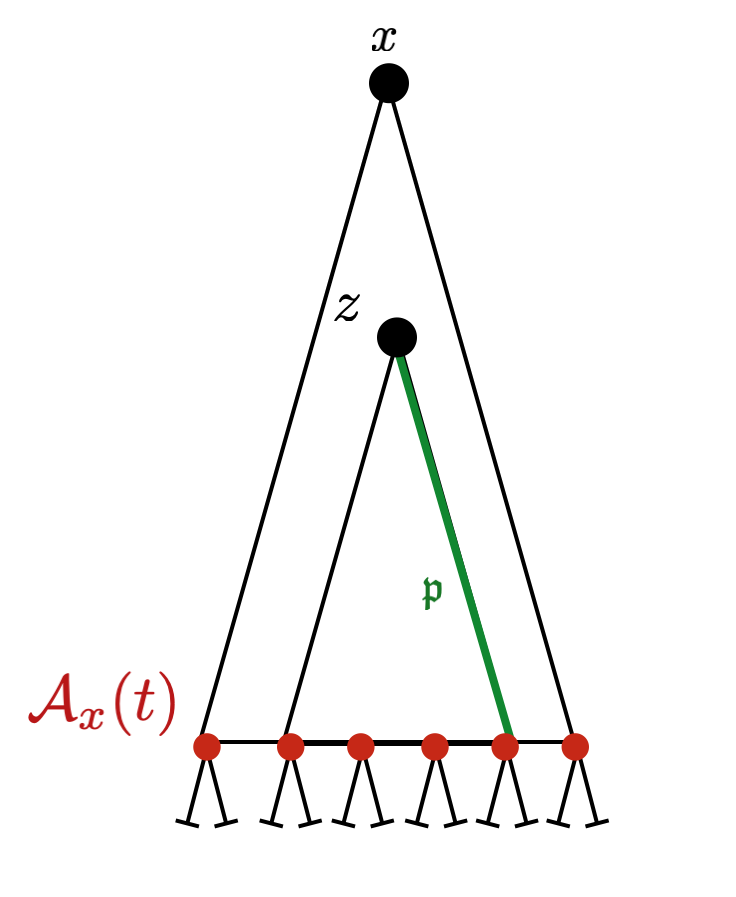}\qquad\includegraphics[width=4.5cm]{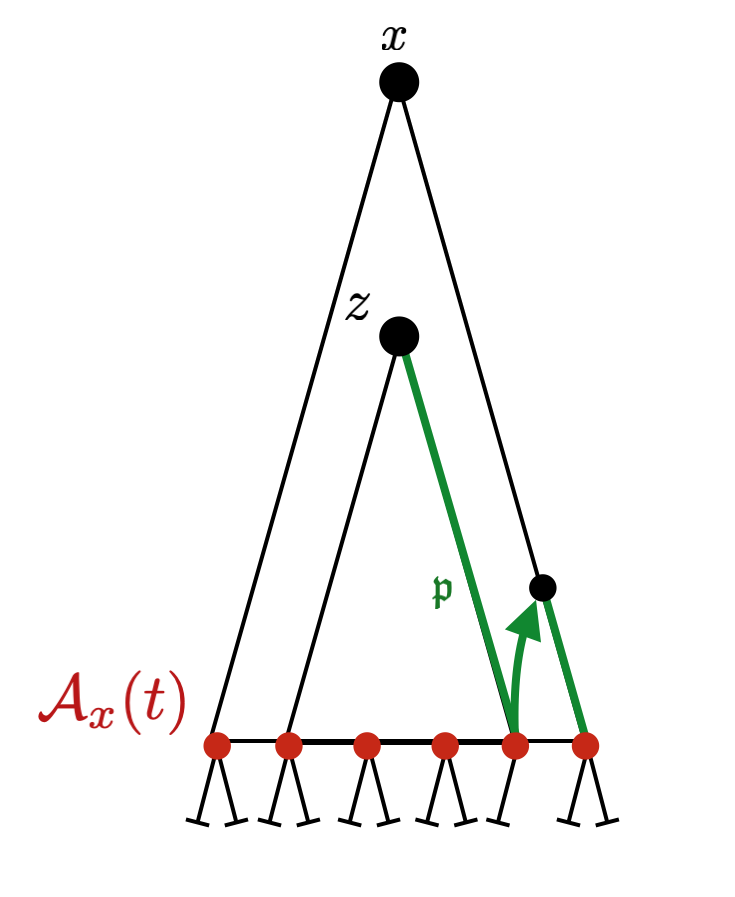}
	\caption{A first visit to $\cA_x(t)$ of a path $\p\in\cP_z(u)$ (left) and an illustration of a bad matching (right).}	\label{fig:figbad2}
\end{figure}

Next, consider an auxiliary  directed tree with random marks $\tilde\cT$ defined as follows: $\tilde\cT$ is a directed regular  tree with deterministic offspring $\Delta$ and height $u$, with independent and identically distributed Bernoulli($p$) marks on its edges, where $p$ is as in \eqref{eq:pp}. Edges whose Bernoulli  mark is $1$ are colored red. A path of length $u$ from the root to one of the leaves is called {\em bad} if it has at least  $K$ red edges. 
	The previous construction then shows that the number of $\p\in\cP_z(u)$ such that  $\cA_x(t)\cap V(\p)> K$ is stochastically dominated by the number of bad paths in $\tilde\cT$.  The probability that a given path in $\tilde\cT$ is bad is given by
	$$
	\P\( \text{Bin}(u,p)\ge K \)\le (up)^K.
	$$
	Therefore, the probability that there exists a bad path in $\tilde\cT$ is at most $\D^u (up)^K$. Since $u\leq \tent\leq \log n /\log 2$, it follows that
	\begin{equation}
	\P\(\exists\: \p\in\cP_z(u),\: \cA_x(t)\cap V(\p)> K    \) \le \Delta ^{T_\ent}(\tent p)^K\le  n^{ \frac{\log\Delta}{\log 2}-K \eta^2/3},
	\end{equation}
	for all $n$ sufficiently large. Taking $K=(9+3\log_2\Delta)/\eta^2$ concludes the proof of \eqref{eq:new1}.
	A union bound then implies that the event $\cE_K$ in the statement of the lemma holds with probability at least $1-1/n$.   
		\end{proof}
	Once Lemma \ref{lemma:new} is available, we can prove Lemma \ref{le:key}. 
	
	\begin{proof}[Proof of Lemma \ref{le:key}]

%
			The distribution $\mu_\l$ satisfies 
			$\mu_\l=\sum_{z\in[n]}\lambda(z)\mu_z$, 
				where  $\mu_z:=\mu_{\d_z}$. 
			Hence, it is sufficient to prove that w.h.p.\
			\begin{equation}
\label{eq:key1z}
\max_{z,x\in[n]}\mu_z(\cA_x(t))\leq \e.
\end{equation} 
We write
			\begin{align}\label{eq:new2}
			\mu_z(\cA_x(t))=\frac1A\sum_{k=0}^{(1-\eta)\tent - t}\a(1-\a)^kQ_z(X_{k+t}\in\cA_x(t)),
			\end{align}
where $Q_z$ is defined as in \eqref{eq:tauze2}.
			As in \cite[Propositon 6]{BCS1} one shows that for both models, with high probability: 
\begin{align}\label{eq:efx2}
\max_{z\in [n]}Q_z(X_{k_0}\in V\setminus V_{*})\leq 2^{-k_0},
\end{align} 
for any fixed constant $k_0$. 
Hence, for all non negative integers $k,t$ with $k_0\leq k+t\leq (1-\eta)\tent$:
			$$
			\max_{z\in [n]}Q_z(X_{k+t}\in\cA_x(t)) \leq  \max_{v\in V_*} Q_v(X_{k+t-k_0}\in\cA_x(t)) + 2^{-k_0}. 
			$$
			Set $u=(1-\eta)\tent$. For any $v\in V_*$ we write
			\begin{align*}
			Q_v(X_{k+t-k_0}\in\cA_x(t))
			&\le 
			Q_v\( (X_0,\dots, X_{u})\not\subset \cT_v(u) \)+\sum_{\p\in \cP_v(u)}Q_v\((X_0,\dots,X_u)=\p,\:X_{k+t-k_0}\in\cA_x(t) \).
			\end{align*}
By \eqref{eq:tauze2}, the first term in the right hand side is w.h.p.\ less than $\e$ uniformly in $v\in V_*$, for any fixed $\e>0$. The second term, taking the summation over $k\in[k_0,(1-\eta)\tent-t]$ satisfies, w.h.p.
\begin{align}
&\sum_{\p\in \cP_v(u)}\sum_{k=k_0}^{(1-\eta)\tent-t} \a(1-\a)^kQ_v\((X_0,\dots,X_u)=\p,\:X_{k+t-k_0}\in\cA_x(t) \)\nonumber\\
&\qquad \leq \a \sum_{\p\in \cP_v(u)}\sum_{k=0}^{u} Q_v\((X_0,\dots,X_u)=\p,\:X_{k}\in\cA_x(t)\)\nonumber\\&
 \qquad \leq \a K \sum_{\p\in \cP_v(u)} Q_v\((X_0,\dots,X_u)=\p\)\leq \a K,
\label{eq:efx3}
\end{align}
where $K$ is the constant from Lemma \ref{lemma:new} and we have used the fact that the event $\cE_K$ from Lemma \ref{lemma:new}  holds with high probability. From \eqref{eq:new2}-\eqref{eq:efx3}, noting that the first $k_0$ terms in the summation over $k$ contribute to \eqref{eq:new2} at most $\a k_0/A$, we then obtain, w.h.p.
 \begin{align}\label{eq:efx4}
\max_{z,x\in[n]}\mu_z(\cA_x(t))\leq \e + \frac{\a k_0}{A} + \frac{\a K}{A} + 2^{-k_0}.
\end{align}
Since $\a\tent\to\g\in(0,\infty]$ and $\a\to 0$, it follows that for all fixed $\eta>0$ one has $\a /A\to 0$ as $n\to\infty$. 
Since the parameters $\e>0$  and $k_0\in\bbN$ are arbitrary this implies the desired conclusion. 
		\end{proof}

\subsection{Proof of Lemma \ref{le:singular}}
Assume $\a\tent \to \g$, and $t=s/\a$, with $s\in(0,\g)$ as in the statement of  Lemma \ref{le:singular}. The proof below applies to both cases $\g\in(0,\infty)$ and $\g=\infty$. We are going to show that for every fixed $\e>0$, there exists an event $\cE_\e=\cE_\e(n)$ such that $\bbP(\cE_\e)\to 1$, $n\to\infty$, and such that on $\cE_\e$ for all $x\in[n]$ there are sets $\cC_{x}\subset V$ 
satisfying
 \begin{align}
 \label{eq:lesingfinal}
\max_{\l\in\cS_n}\max_{x\in[n]}\mu_\l(\cC_x)\leq \e\,,\quad 
\max_{x\in[n]}\pi_0(\cC_x)\leq \e\,,\quad \min_{x\in[n]}P^{t}(x,\cC_x)\geq 1-\e. 
\end{align}
Indeed, using the decompositions in Lemma \ref{le:decomp0} (for the case $\g=\infty$) and in Lemma \ref{le:decomp} (for the case $\g\in(0,\infty)$), if \eqref{eq:lesingfinal} holds then w.h.p.\ 
\begin{align}\label{eq:lesingfin2}
\left\| P^t(x,\cdot) - \pi_{\a,\l}P^t\right\|_{\tv}&\geq \left\| P^t(x,\cdot) - A\mu_\l - (1-A)\pi_0\right\|_{\tv} - \e
\nonumber \\& \geq P^{t}(x,\cC_x) - A\mu_\l(\cC_x) - (1-A)\pi_0(\cC_x) - \e \geq 1-3\e,
\end{align}
where it is understood that $A=1$ if $\g=\infty$. Since \eqref{eq:lesingfin2} holds uniformly in $\l$ and $x$, this completes the proof of Lemma \ref{le:singular}. We turn to the proof of \eqref{eq:lesingfinal}.

It is important that the estimates in \eqref{eq:lesingfinal} hold for all $\e>0$ and $\eta>0$ small enough (but fixed), where $\eta$ is the parameter implicit in the definition of $\mu_\l=\mu_\l^{\eta,t}$. Since $s\in(0,\g)$, we may assume $t\leq (1-2\eta)\tent$ by taking $\eta$ small enough.  
By Theorem \ref{th:BCS}
we know that for each $\d>0$, with high probability there exists sets $\cB_{x}$, such that for all $x\in[n]$: 
\begin{align}\label{eq:lesingfin}
\pi_0(\cB_{x})\leq \d \,,\qquad  P^{t}(x,\cB_{x})\geq 1-\d. 
\end{align}
For $\e>0$, take $k_0=k_0(\e)$ such that $2^{-k_0}\leq\e/2$, and call $t'=t-k_0$. 
Since $t=\Theta(\alpha^{-1})$ and $\alpha\to 0$, we have $t'>0$ for all $n$ large enough. 
For all $x\in[n]$, call $V_{x}$ the subset of vertices in $y\in V_*$ such that $P^{k_0}(x,y)>0$. Define
$$\cC_{x}=\cB_x\cap\left(\cup_{y\in V_{x}} \cA_y(t')\).$$
From \eqref{eq:tauze2} we know that, for all $\d>0$, with high probability,
\begin{equation}\label{eq323}
\min_{y\in V_*} P^{t'}(y,\cA_y(t'))\ge 1-\d.
\end{equation}
By \eqref{eq:efx2}, \eqref{eq:lesingfin} and \eqref{eq323} we obtain 
 \begin{align}\label{eq:lesingfin1}
P^{t}(x,\cC_{x})& \geq P^{t}(x,\cup_{y\in V_{x}} \cA_y(t')) - P^{t}(x,\cB_{x}^c)\nonumber \\ &\geq 
\min_{y\in V_{x}}P^{t'}(y,\cA_y(t')) - 2^{-k_0} - \d\geq 1 -2^{-k_0}- 2\d.
\end{align}
From \eqref{eq:lesingfin} we also know that $\pi_0(\cC_x)\leq \pi_0(\cB_x)\leq\d$. Taking $\d=\e/4$, this and \eqref{eq:lesingfin1} imply the last  two items in \eqref{eq:lesingfinal} since $2^{-k_0}\leq \e/2$. It remains to estimate $\mu_\l(\cC_x)$. Since $\max_{x}|V_{x}|\le \Delta^{k_0}$ we obtain
 \begin{align}
 \label{eq:olesingfinal}
 \mu_\l(\cC_x)\leq \mu_\l(\cup_{y\in V_{x}} \cA_y(t'))\leq \D^{k_0}\max_{y\in[n]}\mu_\l(\cA_y(t')).
\end{align}
From Lemma \ref{le:key} we see that  with high probability, uniformly in $\l$ and $x$, \eqref{eq:olesingfinal} is at most $\D^{k_0}\d$ for any fixed $\d>0$. Thus taking $\d=\D^{-k_0}\e$ concludes the proof of Lemma \ref{le:singular}.  

 \section{Proof of the trichotomy}
\label{sec:trichotomy}
In this section we show how to prove Theorem \ref{th:general} from the facts established above.  
Thus, $G$ is a random graph from either
the directed configuration model  DCM($\bd^\pm$) or the out-configuration model 
OCM($\bd^+$), where the degree sequences satisfy the assumptions \eqref{degs} and \eqref{degs+} respectively, and $\pi_0$ denotes the (w.h.p.) unique stationary distribution for the simple random walk on $G$. 

\subsection{Scenario 1}
We begin with scenario $1$, namely when $\a\tent\to 0$. 
\begin{proposition}\label{prop:ex1}
For 
any sequence $\a$ such that $\a\tent\to0$,
\begin{equation}\label{eq:ex2}
\max_{\l\in\cS_n}\left\| \pi_{\a,\l} - \pi_0\right\|_{\tv}\overset{\P}{\longrightarrow} 0.
\end{equation}
\end{proposition}
\begin{proof}
We need to show that, uniformly in $\l$, for any $\d>0$, 
\begin{equation}\label{eq:ex3}
\|\pi_{\a,\l}-\pi_0\|\leq \d\,,\qquad w.h.p.	
\end{equation}
The upper bound \eqref{eq:expa2} shows that for all $t\in\bbN$:
\begin{equation}\label{eq:ex4}
\|\pi_{\a,\l}-\pi_0\|_\tv\leq(1-(1-\a)^t)+\sum_{k>t}\a(1-\a)^k\|\l P^k-\pi_0\|_\tv.	
\end{equation}
Take $t=s\tent$, with some fixed $s>1$, and observe that by Theorem \ref{th:BCS} we know that for all $k>t$, for all $\l$:
\begin{align}\label{eq:ex5}
\|\l P^k-\pi_0\|_\tv&\leq \|\l P^{s\tent}-\pi_0\|_\tv\nonumber\\
&\leq \max_{x\in V}\|P^{s\tent}(x,\cdot)-\pi_0\|_\tv\leq \d/2\,,\qquad w.h.p.	
\end{align}
In particular, using $\a t\to 0$:
\begin{equation}\label{eq:ex6}
\max_{\l\in\cS_n}\|\pi_{\a,\l}-\pi_0\|_\tv\leq(1-(1-\a)^t)+\d/2\leq \d\,,	\qquad w.h.p.	
\end{equation}
\end{proof}
The claim \eqref{scen1} is thus a consequence of Corollary \ref{cor:la=pi}, Corollary \ref{cor:comp} and Theorem \ref{th:BCS}.


\subsection{Scenario 3}\label{sec:scenario3}
Suppose $\a\tent\to+\infty$, and $t=s/\a$ for some fixed $s\in(0,\infty)$. 
From Lemma \ref{le:singular}, 
%
Proposition \ref{prop:tv} and the upper bound \eqref{eq:ubo} we obtain: 
\begin{equation}\label{eq:singulare2}
\max_{\l\in\cS_n}\max_{x\in [n]}\left| \left\| P_{\a,\l}^t(x,\cdot) - \pi_{\a,\l}\right\|_{\tv} - (1-\a)^t\right|
\overset{\P}{\longrightarrow} 0\,.
\end{equation}
Equivalently, 
\begin{equation}\label{eq:singulare3}
\max_{\l\in\cS_n}\max_{x\in [n]}\left| \cD^x_{\a,\l}(s/\a) - e^{-s}\right|
\overset{\P}{\longrightarrow} 0\,.
\end{equation}
This proves \eqref{scen3}.
\subsection{Scenario 2}\label{sec:scenario2}
Here $\a\tent\to \g\in(0,\infty)$. We take $t=s/\a$, with fixed $s\in(0,\infty)$. We consider separately the case $s\in(\g,\infty)$ and the case $s\in(0,\g)$. 

Suppose first $ s\in(\g,\infty)$.
By Proposition \ref{prop:tv} and the triangle inequality
\begin{align}\label{eq:singulare4}
\cD^x_{\a,\l}(s/\a)
&\leq \left\| P^t(x,\cdot) - \pi_{\a,\l}P^t\right\|_{\tv}\nonumber\\
&\leq \left\| P^t(x,\cdot) - \pi_0\right\|_\tv + \max_{y\in V}\left\| \pi_0-P^t(y,\cdot)\right\|_{\tv}.
\end{align}
Since $ s\in(\g,\infty)$, for some $\e>0$ we have $t\geq (1+\e)\tent$. Therefore, by Theorem \ref{th:BCS} it follows that 
\begin{align}\label{eq:singulare4a}
\max_{\l\in\cS_n}\max_{x\in [n]}\cD^x_{\a,\l}(s/\a)\overset{\P}{\longrightarrow} 0\,,\qquad s \in(\g,\infty).
\end{align}

On the other hand, suppose that $s\in(0,\g)$. Here we can apply Lemma \ref{le:singular}, Proposition \ref{prop:tv} and the upper bound \eqref{eq:ubo},  as in Section \ref{sec:scenario3} above,
to obtain 
\begin{equation}\label{eq:singulare23}
\max_{\l\in\cS_n}\max_{x\in [n]}\left| \cD^x_{\a,\l}(s/\a) - e^{-s}\right|
\overset{\P}{\longrightarrow} 0\,,\qquad s\in(0,\g).
\end{equation}
Combining \eqref{eq:singulare4a} and \eqref{eq:singulare23}, we have proved \eqref{scen2}.


\section{Widespread measures}\label{sec:martingales}
 The goal of this section is to prove Lemma  \ref{le:approx}. We remark that the statement $\|\pi_{\a,\l}-\pi_0\|_\tv\to 0$ in probability is a consequence of \eqref{eq:approx}. 
Indeed, 
fix any sequence $\a=\a(n)\to 0$, and take 
$t=t(n)\to\infty$ such that $\a t \to 0$. From \eqref{eq:approx} we know that 
\begin{equation}\label{eq:approx22}
\left\|\l P^t - \pi_0\right\|_{\tv}\overset{\P}{\longrightarrow} 0.
		\end{equation}
		As in \eqref{eq:ex4}, from the upper bound \eqref{eq:expa2} and the monotonicity in time of total variation distance to stationarity we obtain:
\begin{equation}\label{eq:approx23}
\|\pi_{\a,\l}-\pi_0\|_\tv\leq(1-(1-\a)^t)+\|\l P^t-\pi_0\|_\tv.	
\end{equation}
Using \eqref{eq:approx22} and $\a t \to 0$ we conclude the proof.
Thus, we are left to prove \eqref{eq:approx}.

In the special case where $\l=\muin$, and for the directed configuration model DCM($\bd^\pm$), a similar result was already obtained in \cite{BCS1}. Here we are going to prove it 
for the case of the out-configuration model 
OCM($\bd^+$) as well, and more importantly we are going to extend it to the case of an arbitrary widespread probability measure $\l$.
Following the approach in \cite{BCS1}, the proof of Lemma \ref{le:approx} will be based on the construction of a martingale approximation for the distribution $\l P^t$. The latter, in turn, rests on a branching approximation which allows one to couple the in-neighbourhood of a uniformly distributed random vertex of $G$ with a marked Galton-Watson tree up to depth $t=o(\log n)$.

 We start with the definition of the relevant branching processes and the associated martingales. 
These will later be used in a coupling argument to provide an approximate description of the in-neighbourhood of a vertex in our random graphs, and of the stationary distribution at that vertex. 
Since the constructions differ slightly for the two models DCM($\bd^\pm$) or  
OCM($\bd^+$) we will define two distinct random trees $\cT^-(\bd^\pm)$ and $\cT^-(\bd^+)$.

\subsection{The marked Galton-Watson trees $\cT^-(\bd^\pm)$, $\cT^-(\bd^+)$}\label{sec:gwt}
Given $n\in\bbN$, and a double sequence $ \bd^\pm$ of degrees satisfying \eqref{degsm}
 and \eqref{degs}, for each $i\in [n]$, we define the rooted random marked tree $\cT^-_i(\bd^\pm)$ recursively with the following rules:
\begin{itemize}
\item
the root is given the mark $i$;
\item
every vertex with mark $j$ has $d^-_j$ children, each of which is given independently the mark $k\in[n]$ with probability $d^+_k/m$.
\end{itemize}
On the other hand, given $n\in\bbN$, and a sequence $ \bd^+$ of degrees satisfying \eqref{degs+}, for each $i\in [n]$, the rooted random marked tree $\cT^-_i(\bd^+)$ is defined by:
\begin{itemize}
\item
the root is given the mark $i$;
\item
regardless of its own mark every vertex has, for each $j\in[n]$ independently with probability $d^+_j/n$,  a child with mark $j$.
\end{itemize}
There are several differences between the two trees $\cT^-_i(\bd^\pm)$ and $\cT^-_i(\bd^+)$. In the first case the number of children of a given vertex is a deterministic function of the vertex's mark, whereas in the second case it is a random variable $D$ that can be written as 
	\begin{equation}\label{eq:randeg}
D = \sum_{j\in[n]}Y_j\,,\qquad Y_j={\rm Ber}(d^+_j/n),
	\end{equation}
	where the $Y_j$ are independent Bernoulli random variables with parameters $d^+_j/n$. In particular, the average number of children of any given vertex in $\cT^-_i(\bd^+)$ is 
		\begin{equation}\label{eq:arandeg}
\bbE[D] = \sum_{j\in[n]}\frac{d^+_j}{n}=\frac{m}n=\left\langle d\right\rangle.
	\end{equation}
  Since $D$ can be zero, in contrast with the tree $\cT^-_i(\bd^\pm)$, the tree $\cT^-(\bd^+)$ is finite with positive probability. However, the two trees share several common features and we shall try to treat the two cases in a unified fashion as much as possible. 

We write $\bo$ for the root and $\bx,\by$
for other vertices of the tree, with the notation $\by\to \bx$ if $\by$ is a child of $\bx$. Each vertex $\bx$ of the tree has a mark, which we denote by $i(\bx)$.  
If $\cI$ denotes an  independent uniformly random  $i\in[n]$, and the root is given the mark $i(\bo)=\cI$, then we write $\cT^-(\bd^\pm)=\cT^-_\cI(\bd^\pm)$ and $\cT^-(\bd^+)=\cT^-_\cI(\bd^+)$. Notice that $\cT^-(\bd^\pm)$ and $\cT^-(\bd^+)$ have the same average degree at the root, given by \eqref{eq:arandeg}. We often write $\cT^-$ for short if this creates no confusion. For each $t\in\bbN$ we let $\cT^{-,t}$ denote the set of vertices in the generation $t$ of the tree.
Each vertex $\bx\in\cT^{-,t}$ has a unique path $(\bx_t,\bx_{t-1},\dots,\bx_1,\bx_0)$ connecting it to the root with $\bx_t=\bx$ and $\bx_0=\bo$. To any such $\bx$ we associate the weight
 	\begin{equation}\label{eq:weight}
w(\bx)=
\prod_{u=1}^t\frac1{d^+_{i(\bx_u)}}.
	\end{equation}
	If $\cT^{-,t}$ coincides with the in-neighbourhood of $\bo$ in a digraph $G$, then $w(\bx)$ is the probability that the simple random walk on $G$ goes from $\bx$ to $\bo$ in $t$ steps.  
	\subsection{Martingale approximation}
 Given a function $\varphi:[n]\mapsto \bbR$, we define the process
 	\begin{equation}\label{eq:proc1}
X_t(\varphi) = \sum_{\bx\in\cT^{-,t}}\varphi(i(\bx))w(\bx),\qquad X_0(\varphi) = \varphi(i(\bo)).
	\end{equation}
	  We write $\cF_t$ for the $\si$-algebra generated by the random tree $\cT^-$ up to and including generation $t$. 

\begin{lemma}\label{le:GW1}
Let $\cT^-$ be either $\cT^-(\bd^\pm)$ or $\cT^-(\bd^+)$, and write $\bar\varphi=\sum_{j=1}^n\varphi(j).$ Then, for all $t\in\bbN$:
 \begin{equation}\label{eq:proc2}
\bbE[X_t(\varphi)|\cF_{t-1}] = X_{t-1}(\bar\varphi\muin).
	\end{equation}
\end{lemma}
\begin{proof}
Let the symbol $\sum_{\by\to \bx}$ denote the sum over the set of children of $\bx$ and note the symbolic identity
	\begin{equation}\label{eq:new}
	\sum_{\by\in\cT^{-,t}}\equiv\sum_{\bx\in\cT^{-,t-1}}
	\sum_{\by\to\bx }.
	\end{equation}
Therefore, 
\begin{align}\label{eq:proc3}
\bbE[X_t(\varphi)|\cF_{t-1}] &=\sum_{\bx\in\cT^{-,t-1}}\bbE\left[\sum_{\by\to\bx }\varphi(i(\by))w(\by)|\cF_{t-1}\right]\nonumber\\
&= \sum_{\bx\in\cT^{-,t-1}}w(\bx)\,\bbE\left[\sum_{\by\to\bx }\frac{\varphi(i(\by))}{d^+_{i(\by)}}\,\Big|\,\cF_{t-1}\right].
	\end{align}
	For the tree $\cT^-(\bd^\pm)$ we have
	\begin{align}\label{eq:proc4}
\bbE\left[\sum_{\by\to\bx }\frac{\varphi(i(\by))}{d^+_{i(\by)}}\,\Big|\,\cF_{t-1}\right]
=
d^-_{i(\bx)}\sum_{j=1}^n\frac{d^+_j}{m}\frac{\varphi(j)}{d^+_j}=\bar\varphi\,\muin(i(\bx)).
	\end{align}
	For the tree $\cT^-(\bd^+)$ we have
	\begin{align}\label{eq:proc5}
\bbE\left[\sum_{\by\to\bx }\frac{\varphi(i(\by))}{d^+_{i(\by)}}\,\Big|\,\cF_{t-1}\right]
=
\sum_{j=1}^n\frac{d^+_j}{n}\frac{\varphi(j)}{d^+_j}=\bar\varphi\,\muin(i(\bx)).
	\end{align}
	This proves \eqref{eq:proc2}. 
	\end{proof}

	In particular, when $\varphi=\muin$, then 
	$$
	\bbE[X_t(\muin)|\cF_{t-1}] = X_{t-1}(\muin)\,,\qquad t\in\bbN.
	$$
Therefore, $X_t(\muin)$ is a martingale with respect to the filtration $\cF_t$. It is convenient to normalize it and consider instead the martingale defined as
	\begin{align}\label{eq:defmar}
M_t = nX_t(\muin)=\sum_{\bx\in\cT^{-,t}}n\muin(i(\bx))w(\bx),\qquad M_0 = n\muin(i(\bo)).
	\end{align}
	Notice that $\bbE[M_t]=\bbE[M_0] = n\bbE[\muin(\cI)]=1$. 
In the case of model 1, the following convergence result was already discussed in  \cite[Proposition 15]{BCS1}.
	\begin{proposition}
\label{prop:marconv}
For every fixed $n$, as $t\to\infty$ the martingale $M_t$ converges to a limit $M_\infty$, both almost surely and in $L^2$ (see \cite[Ch. 12]{Williams}) and for all $t\in\bbN$:
	\begin{align}\label{eq:mar1}
\bbE[(M_t-M_\infty)^2] = C\r^t\,
	\end{align}
where the constants $\r,C$ are given by 
\begin{equation}\label{eq:ac}
\r=\sum_{j=1}^n\muin(j)\,\frac1{d^+_j}\,,\qquad 
C = \begin{cases}
\frac{n}{m(1-\r)}\sum_{j=1}^n
\frac{(d^-_j-d^+_j)^2}{md^+_j}  & \text{{\rm \small model 1}}
\\
\frac{\r-1/n}{1-\r}  & \text{{\rm \small model 2}}
\end{cases}
\end{equation}
	\end{proposition}
	\begin{proof}
	Consider the increments  
		\begin{align}\label{eq:mar2}
\D_t=M_{t+1}-M_t .
	\end{align}
Reasoning as in Lemma \ref{le:GW1}, for both models we write 
	
	\begin{align}\label{eq:mar2bis2}
	\D_t=\sum_{\bx\in\cT^{-,t}}n\muin(i(\bx))w(\bx)\psi(\bx),
	\end{align}
	where $\psi$ is defined as
	\begin{align}\label{eq:mar2bis}
\psi(\bx)
	=\sum_{\by\to\bx }\frac{\muin(i(\by))}{\muin(i(\bx))d^+_{i(\by)}} - 1.
	\end{align}
As in Lemma \ref{le:GW1} one has $\bbE[\psi(\bx)\tc\cF_t]=0$. Let us compute $\bbE[\psi(\bx)^2|\,\cF_t]$.
For the tree $\cT^-(\bd^\pm)$, we can rewrite
	\begin{align}\label{eq:mar3}
\psi(\bx)
= \sum_{\by\to\bx }\left[\frac{\muin(i(\by))}{\muin(i(\bx))d^+_{i(\by)}} - \frac1{d^-_{i(\bx)}}\right].
\end{align}
Therefore,
	\begin{align}\label{eq:mar4}
\bbE[\psi(\bx)^2|\,\cF_t]& 
=
d^-_{i(\bx)}\sum_{j=1}^n\frac{d^+_j}{m}\left(\frac{d^-_j}{d^-_{i(\bx)}d^+_j}-\frac1{d^-_{i(\bx)}}\right)^2=\frac{C_1}{d^-_{i(\bx)}}
\,,
	\end{align}
where we use the notation  
$$
C_1=\sum_{j=1}^n\frac{(d^-_j-d^+_j)^2}{md^+_j} .
$$
For the tree $\cT^-(\bd^+)$ we have
	\begin{align}\label{eq:mar5}
\bbE[\psi(\bx)^2\tc\cF_t]& 
= \bbE\left[\left(\sum_{\by\to\bx }\frac1{d^+_{i(\by)}}\right)^2-2\sum_{\by\to\bx }\frac1{d^+_{i(\by)}} + 1\,\Big|\,\cF_t\right]\nonumber\\&=
\sum_{j\neq j'}\frac{d^+_jd^+_{j'}}{n^2}\frac{1}{d^+_{j}d^+_{j'}} + \sum_{j}\frac{d^+_j}{n}\frac1{(d^+_j)^2}- 
2\sum_{j}\frac{d^+_j}{n}\frac1{d^+_j}+1 = \r-\frac1n
\,,
	\end{align}
where $\rho$ is as in \eqref{eq:ac}. 
Since $\bbE[\psi(\bx)\psi(\bx')\tc\cF_t]=0$ for all $\bx,\bx'\in\cT^{-,t}$ with $\bx\neq \bx'$, 
\begin{align}\label{eq:mar6}
\bbE[\D_t^2\tc\cF_t]& = \sum_{\bx\in\cT^{-,t}}n^2\muin(i(\bx))^2w(\bx)^2
\bbE[\psi(\bx)^2\tc\cF_t]. 
\end{align}
Therefore, combining \eqref{eq:mar4} and \eqref{eq:mar5} we have 
\begin{align}\label{eq:mar7}
\bbE[\D_t^2\tc\cF_t]& =  C(1-\r)\sum_{\bx\in\cT^{-,t}}n\muin(i(\bx))w(\bx)^2\,,
\end{align}
where $\r,C$ are given by  \eqref{eq:ac}. Furthermore, observe that in both models one has
\begin{align}\label{eq:mar8}
\bbE[\D_t^2\tc\cF_{t-1}]& = \bbE\left[\bbE[\D_t^2\tc\cF_{t}]\tc\cF_{t-1}\right]\nonumber\\
&= C(1-\r)\sum_{\bx\in\cT^{-,t-1}}n\muin(i(\bx))w(\bx)^2 \bbE\left[\sum_{\by\to \bx}\frac{\muin(i(\by))}{\muin(i(\bx))(d^+_{i(\by)})^2}\tc\cF_{t-1}\right]\nonumber\\ & = C(1-\r)\r\sum_{\bx\in\cT^{-,t-1}}n\muin(i(\bx))w(\bx)^2=\r\,\bbE[\D_{t-1}^2\tc\cF_{t-1}]. 
\end{align}
Thus, iterating we obtain
\begin{align}\label{eq:mar9}
\bbE[\D_t^2]
= \bbE[\D_{0}^2]\r^t = C(1-\r)\bbE[n\muin(\cI)] \r^t = C(1-\r)\r^t. 
\end{align}
Since $d^+_j\geq 2$ one has $\r\leq 1/2$. Thus $M_t$ is a martingale bounded in $L^2$, and therefore $M_t\to M_\infty$ almost surely and in $L^2$, for some $M_\infty\in L^2$. Using the orthogonality $\bbE[\D_t\D_{t'}]=0$ for all $t\neq t'$, \eqref{eq:mar1} follows by summing \eqref{eq:mar9} from $t$ to $+\infty$.  
	\end{proof}
	\begin{remark}
For each fixed $n\in\bbN$, one can characterise the random variable $M_\infty$ as the solution to a distributional fixed point equation.
For the directed configuration model DCM($\bd^\pm$) this is discussed in \cite[Lemma 16]{BCS1}. With a similar reasoning, for the out-configuration model OCM($\bd^+$) one obtains that 
\begin{align}\label{eq:rde}
M_\infty
\overset{d}{=}
\sum_{j=1}^n\frac{Y_j}{d^+_j} \,M_{\infty,j},
\end{align}
where $\overset{d}{=}$ stands for equality of distributions, $M_{\infty,j}$ are i.i.d.\ copies of $M_\infty$ and $Y_j$ are independent Bernoulli random variables with parameter $d^+_j/n$.  
\end{remark}
The next result will be crucial for the analysis of widespread measures. Notice that the constant $\g(\l)$ appearing in the estimate below is bounded uniformly in $n$ if and only if $\l$ satisfies \eqref{def:ws2}. 

	\begin{proposition}\label{prop:mara}
For any probability vector $\l$, and any $t\in\bbN$:
 	\begin{align}\label{eq:mara1}
\bbE[(M_t-nX_t(\l))^2] \leq \g(\l)\r^t\,,
	\end{align}
	where $\r\in(0,1)$ is as in Proposition \ref{prop:marconv} and $\g(\l)$ is defined as
	\begin{equation}\label{eq:mara2}
\g(\l) = \frac{n}{2}\sum_{j=1}^n(\l(j)-\muin(j))^2
\end{equation}
	\end{proposition}
	\begin{proof}
Setting $\varphi(j)=n(\muin(j)-\l(j))$, we write $M_t-nX_t(\l)=X_t(\varphi)$. Since $\bar\varphi=0$, Lemma \ref{le:GW1}
shows that $\bbE[M_t-nX_t(\l)|\cF_{t-1}] =0$. We now compute 
$$\G_{t}:=\bbE[(M_{t+1}-nX_{t+1}(\l))^2|\cF_{t}].$$
Using $\bar\varphi=0$ one has
\begin{align}\label{eq:mara3}
\G_t&=\bbE[X_{t+1}(\varphi)^2|\cF_{t}]
\nonumber\\&
=
\sum_{\bx\in\cT^{-,t}}w(\bx)^2
\,\bbE\left[\left(\sum_{\by\to\bx }\frac{\varphi(i(\by))}{d^+_{i(\by)}}\right)^2\Big|\,\cF_{t}\right].
	\end{align}
	For the tree $\cT^-(\bd^\pm)$ we have
	\begin{align}\label{eq:mara4}
\bbE\left[\left(\sum_{\by\to\bx }\frac{\varphi(i(\by))}{d^+_{i(\by)}}\right)^2\Big|\,\cF_{t}\right]
=
d^-_{i(\bx)}\sum_{j=1}^n\frac{d^+_j}{m}\frac{\varphi(j)^2}{(d^+_j)^2}=\muin(i(\bx))\sum_{j=1}^n\frac{\varphi(j)^2}{d^+_j}.
	\end{align}
On the other hand for the tree $\cT^-(\bd^+)$ we have
	\begin{align}\label{eq:mara5}
\bbE\left[\left(\sum_{\by\to\bx }\frac{\varphi(i(\by))}{d^+_{i(\by)}}\right)^2\Big|\,\cF_{t}\right]
&=
\sum_{j\neq j'}\frac{\varphi(j)\varphi(j')}{n^2} + \sum_{j}\frac{\varphi(j)^2}{nd^+_j}
\nonumber\\ & =  
\frac1n\sum_{j=1}^n\frac{\varphi(j)^2}{d^+_j} \left(1 - \frac{d^+_j}n\right) 
	\end{align}
Summarising, we have shown that
 \begin{equation}\label{eq:mara6}
\G_t= C(\l)\sum_{\bx\in\cT^{-,t}}n\muin(i(\bx))w(\bx)^2\,,\qquad 
C(\l)=\frac1n
\begin{cases}\sum_{j=1}^n\frac{\varphi(j)^2}{d^+_j} & \text{{\rm \small model 1}}
\\
\sum_{j=1}^n\frac{\varphi(j)^2}{d^+_j} \left(1 - \frac{d^+_j}n\right)  & \text{{\rm \small model 2}}
\end{cases}
\end{equation}
Thus, the same argument used in \eqref{eq:mar8} implies that 
in both models 
\begin{align}\label{eq:mara7}
\bbE[\G_t\tc\cF_{t-1}]& = \r\G_{t-1},
\end{align}
where $\rho$ is defined as in \eqref{eq:ac}. Therefore, 
\begin{align}\label{eq:mara8}
\bbE[\G_t] = \bbE[\G_{0}]\r^t = C(\l)\bbE[n\muin(\cI)]\r^t= C(\l)\r^t.
\end{align}
The desired bound follows from the fact that in both models $C(\l)\leq\g(\l)$. 
	\end{proof}
 
 \subsection{Branching approximation for in-neighbourhoods}\label{brax}
The $t$-in-neighbourhood of a vertex $v$, denoted $B^-_{v,t}$,  is defined as the subgraph of $G$ induced by the set of directed paths of length $t$ in $G$ which terminate at vertex $v$. Here we observe that for any fixed $v\in[n]$, if $t$ is a small multiple of $\log n$ then with high probability $B^-_{v,t}$ can be coupled to the first $t$ generations of the random trees defined in Section \ref{sec:gwt}. 
We consider the two models separately.

\subsubsection{$B^-_{v,t}$ for DCM($\bd^\pm$)}\label{sec:generat}
Recall that each vertex $x$ has $d^-_x$ heads and $d^+_x$ tails. Call $E_x^-$and $E_x^+$ the sets of heads and tails at $x$ respectively. The uniform bijection $\o$ between heads and tails, viewed as a matching, 
can be sampled by iterating the  following steps until there are no unmatched heads left:
\begin{enumerate}[1)]
\item pick an unmatched head $e_-$ according to some priority rule;
\item pick an unmatched tail $e_+$ uniformly at random;
\item match $e_-$ with $e_+$, i.e.\ set $\o(e_+)=e_-$.
\end{enumerate}
Note that this gives the desired uniform distribution over matchings regardless of the priority rule chosen at step 1. The graph $G$ is obtained by adding a directed edge $(x,y)$ whenever $e_-\in E_y^-$ and $e_+\in E_x^+$ in step 3 above. 

To generate $B^-_{v,t}$ only, one can start at vertex $v$ and run the previous sequence of steps, by giving priority to those unmatched heads which have minimal distance from vertex $v$, until this minimal distance exceeds $t$, at which point the process stops. During the process, say that a vertex $x$ is {\em exposed} if at least one of the tails $e_+\in E_x^+$ or heads $e_-\in E_x^-$ has been already matched. Notice that as long as in step 2 no tail $e_+$ is picked from exposed vertices, the resulting digraph is a directed tree. 

Let us now describe a coupling of the in-neighbourhood 
 $B^-_{v,t}$ and the marked tree $\cT^-_{v,t}(\bd^\pm)$, where $\cT^-_{v,t}(\bd^\pm)$ stands for the marked tree $\cT^-_v(\bd^\pm)$  up to generation $t$; see Section \ref{sec:gwt} for the definition of $\cT^-_v(\bd^\pm)$.
 Clearly, step $2$ above can be modified by picking $e$ uniformly at random among all (matched or unmatched) tails and rejecting the proposal if the tail was already matched.
The tree can then be generated by iteration of the same sequence of steps with the difference that at step $2$ we  
never reject the proposal 
and at step $3$ we add a new leaf 
to the current tree, with mark $x$ if $e_+\in E_x^+$, together with a new set of $d^-_x$ unmatched heads attached to it.   
Call $\t$ the first time that a uniform random choice among all tails gives $e_+\in E_x^+$ with $x$ already 
in the tree. By construction,   the in-neighbourhood and the tree coincide up to time $\t$.
At the $k$-th iteration, the probability of picking a tail with a mark already used is at most $k\D/m$, where $\D$ is the maximum degree. 
Therefore, by a union bound,
\begin{align}\label{eq:coup1}
\bbP(\t\leq k) \leq \frac{k^2\D}{m}. 
\end{align}
Taking $k=\D^{t+1}$ steps, we have necessarily uncovered the whole in-neighbourhood 
 $B^-_{v,t}$. Thus, we have proved the following statement.
 \begin{lemma}\label{le:coup1}
 The $t$-in-neighbourhood 
 $B^-_{v,t}$ and the marked tree $\cT^-_{v,t}(\bd^\pm)$ can be coupled in such a way that 
 \begin{align}\label{eq:coup12}
\bbP\left(B^-_{v,t}\neq\cT^-_{v,t}(\bd^\pm)\right) \leq \frac{\D^{2t+3}}{m}. 
\end{align}
 \end{lemma}

\subsubsection{$B^-_{v,t}$ for OCM($\bd^+$)}\label{sec:inneighocm}
 Recall that each vertex $x$ has $d^+_x$ tails, and call $E_x^+$ the sets of tails at $x$. 
 Consider the following {\em exploration process} of the in-neighbourhood at a fixed vertex $v$. 
 The process is defined as a triple $(\cC_\ell,\cA_\ell,\phi_\ell)$ where $\cC_\ell,\cA_\ell\subset [n]$ are respectively the {\em completed} set and the {\em active} set at time $\ell$, and $\phi_\ell:[n]\mapsto\bbZ_+$ is a map such that $\phi_\ell(y) \in\{0,\dots,d^+_y\}$ for each $y\in[n]$, $\ell\in\bbZ_+$.
 At time zero we set $\cC_0=\emptyset, \cA_0=\{v\}$, and $\phi_0(y)=0$ for all $y\in[n]$. The $\ell$-th iteration of the exploration determines the triple $ (\cC_\ell,\cA_\ell,\phi_\ell)$ by executing the following steps:
 \begin{enumerate}[1)]
\item pick a vertex $x\in\cA_{\ell-1}$ according to some priority rule;
\item for each $y=1,\dots,n$ independently, sample $X_{\ell,y}$ defined as the Bernoulli random variable with parameter
	\begin{equation}\label{eq:def-pyl}
	p_{\ell,y}=\frac{d^+_y-\phi_{\ell-1}(y)}{n-\ell+1},
	\end{equation}
 call $V_\ell$ the set of $y\in[n]$ such that $X_{\ell,y}=1$, and define
 \begin{equation}
W_\ell=(\cC_{\ell-1}\cup\cA_{\ell-1})^c\cap V_\ell.
\end{equation}
\item define the new triple $(\cC_\ell,\cA_\ell,\phi_\ell)$ as $$
\cC_\ell=\cC_{\ell-1}\cup\{x\}\,,\quad \cA_\ell=\cA_{\ell-1}\setminus\{x\}\cup W_\ell\,,\quad\phi_\ell(y)=\phi_{\ell-1}(y)+\ind(y\in V_\ell), \;y=1,\dots,n.$$
\end{enumerate}
Note that this process stops when $\cA_\ell$ becomes empty. Let us call $\t_\emptyset$ this random time:
 \begin{align}\label{eq:teset}
\t_\emptyset=\min\{\ell\geq 1: \; \cA_\ell=\emptyset\}.
\end{align}
For instance,  $\t_\emptyset=1$ with probability $\prod_{y=1}^n(1- d^+_y/n)$. We may construct a digraph $G_v(\ell)$ along with the above process by adding the directed edges $(y,x)$ for all $y\in V_\ell$ at step $2$.  Notice that when the process stops $G_v(\t_\emptyset)$ is a sample of the subgraph of $G$ induced by all directed paths in $G$ that
terminate at $v$. In particular, if the priority in step 1 is given to $x$ which have minimal distance to $v$, and if we stop the process as soon as all active vertices have distance to $v$ larger than $t$ in the current graph $G_v(\ell)$, we obtain the in-neighbourhood of $v$ at distance $t$, namely the digraph $B^-_{v,t}$ for the model OCM($\bd^+$). More formally, if $\t_t$ denotes the minimal $\ell$ such that all $x\in\cA_{\ell}$ have distance to $v$ at least $t+1$ in $G_v(\ell)$ then, $B^-_{v,t}$ is given by the subgraph of $ G_v(\t_t\wedge\t_\emptyset)$ induced by the completed set $\cC_{\t_t\wedge\t_\emptyset}$, where $a\wedge b $ denotes the minimum of $a,b$.

	Let us remark that the quantity $p_{\ell,y}$ in \eqref{eq:def-pyl} cannot exceed $1$. In fact, in case there exists some $\ell\in\N$ such that $p_{\ell,y}=1$ then it means that at most $d_y^+$ vertices need to be discovered at step $\ell$, and vertex $y$ needs to link to all of them. Hence, $p_{\ell,y}$ stays $1$ up to the end of the process.

Let us now describe a coupling of 
 $B^-_{v,t}$ and the marked tree $\cT^-_{v,t}(\bd^+)$, where we write $\cT^-_{v,t}(\bd^+)$ for the marked tree $\cT^-_v(\bd^+)$ up to generation $t$; see Section \ref{sec:gwt}.
 First, observe that the tree $\cT^-_v(\bd^+)$ is obtained by iterating the steps above with the difference that at step 2 the probability $p_{\ell,y}$ must be taken always equal to $d^+_\ell/n$, and that each $y\in V_\ell$ yields a new child 
 with mark $y$ in the current tree. Let $\cT^-_v(\ell)$ denote the tree obtained after $\ell$ iterations, and let $\D=\max_xd^+_x$.
 \begin{lemma}\label{le:ell}
 The random variables $G_v(\ell),\cT^-_v(\ell)$
can be coupled in such a way that for every $\ell\in\bbN$:
  \begin{align}\label{eq:coup21}
  \bbP(G_v(\ell)\neq \cT^-_v(\ell))\leq \frac{\D^2\ell^2}{n-\ell}.
 \end{align}
 \end{lemma}
\begin{proof}
Let $E_\ell=\{G_v(\ell)\neq \cT^-_v(\ell)\}$. Since at time 0 one has $G_v(0)=\cT^-_v(0)=\{v\}$, the event $E_\ell$ satisfies $E_\ell=\cup_{k=1}^\ell E_{k-1}^c \cap E_k$, so that 
  \begin{align}\label{eq:coup22}
  \bbP\left(G_v(\ell)\neq \cT^-_v(\ell)\right)\leq \sum_{k=1}^\ell\bbP( E_{k-1}^c \cap E_k)
 \end{align}
Consider now the $k$-th iteration, and assume that  $G_v(k-1)= \cT^-_v(k-1)$. Thus, we may pick the same $x$ in step 1 for both samples. 
At step 2, let $X_{k,y}$ denote the Bernoulli random variables with parameter $p_{k,y}$ used for the sampling of $G_v(k)$ and let $\tilde  X_{k,y}$ be the Bernoulli random variables with parameter $d^+_y/n$ used for the sampling of $\cT^-_v(k)$. The total variation distance between two Bernoulli random variables equals the absolute value of the difference of their parameters. Therefore, for each $y$ independently we may couple $(X_{k,y},\tilde  X_{k,y})$ with probability $1-|p_{k,y}-d^+_y/n|$. Notice that
if $G_v(k)\neq \cT^-_v(k)$, then either at least one of the pairs $(X_{k,y},\tilde  X_{k,y})$ fails to couple, or at least one of the $y\in \cC_{k-1}\cup \cA_{k-1}$  has $\tilde X_{k,y}=1$. Thus, on the event $ E_{k-1}^c$, the probability of $E_k$ given the history up to the $(k-1)$-th iteration is bounded above by
%
\begin{align}\label{eq:aaaa}
\sum_{y\not\in \cC_{k-1}\cup\cA_{k-1}}|p_{k,y}-d^+_y/n| + \sum_{y\in \cC_{k-1}\cup\cA_{k-1}} p_{k,y}.
\end{align}
If $y\not\in \cC_{k-1}\cup\cA_{k-1}$, then $\phi_{k-1}(y)=0$ and $p_{k,y}-d^+_y/n=\frac{d_y^+}{n}\cdot\frac{k-1}{n-k+1}$.  For the second term we write  $|\cC_{k-1}\cup\cA_{k-1}|\leq Z_{k-1}$, 
where $Z_\ell$ denotes the number of edges in the tree $\cT^-_v(\ell)$.  In conclusion, \eqref{eq:aaaa} is bounded by 
\begin{align*}
 \frac\D{n-k}\left(k+Z_{k-1}\right) .
\end{align*}
Thus, letting $\cF_\ell$ denote the $\si$-algebra generated by the two processes up to time $\ell$, we have obtained
\begin{align}\label{eq:coup23}
 \bbP( E_{k-1}^c \cap E_k)& = \bbE\left[\bbE\left[\ind(E_{k-1}^c \cap E_k)\tc\cF_{k-1}\right]\right]\nonumber\\
&\leq \frac\D{n-k}\left(k+\bbE[Z_{k-1}]\right).
 \end{align}
From \eqref{eq:arandeg} we deduce $\bbE[Z_{k-1}]=(k-1)\left\langle d\right\rangle\leq (k-1)\D$. Therefore, the estimate \eqref{eq:coup21} follows from \eqref{eq:coup22} and \eqref{eq:coup23}.
\end{proof}  
The next lemma establishes the coupling estimate for the $t$-in-neighbourhood
 $B^-_{v,t}$ and the tree $\cT^-_{v,t}(\bd^+)$. The estimate could be refined but \eqref{eq:coup31} below will be more than sufficient for our purposes.  
 \begin{lemma}\label{le:coup3}
 The random variables $B^-_{v,t}$ and the tree $\cT^-_{v,t}(\bd^+)$
can be coupled in such a way that for every $t\leq \frac{\log n}{4\log \D}$, for all $n$ large enough:
  \begin{align}\label{eq:coup31}
  \bbP\left(B^-_{v,t}\neq \cT^-_{v,t}(\bd^+)\right)\leq \frac{\D^{3t}(\log n)^4}{n}.
 \end{align}
 \end{lemma}
\begin{proof}
Let $|\cT^-_{v,t}|$ denote the number of edges in the tree $\cT^-_{v,t}=\cT^-_{v,t}(\bd^+)$. Since at each iteration the number of edges added is stochastically dominated by a binomial random variable with parameters $n$ and $\D/n$, one has a large deviation bound for $ |\cT^-_{v,t}|$ of the form: there exist absolute constants $a,A>0$ such that 
  \begin{align}\label{eq:coup32}
  \bbP\left(|\cT^-_{v,t}|> s\D^t \right)\leq A\,e^{-a\,s}, \qquad s\geq 1.
 \end{align}
 The estimate \eqref{eq:coup32} can be proved e.g.\ by repeating the argument in  \cite[Lemma 23]{BLM}. Next, observe that if $|\cT^-_{v,t}|\leq s\D^t$ and $B^-_{v,t}\neq \cT^-_{v,t}$, then there must exist $\ell=1,\dots,s\D^t$ such that $G_v(\ell)\neq \cT^-_v(\ell)$. The latter probability can be bounded via Lemma \ref{le:ell}. Summarizing,
  \begin{align}\label{eq:coup33}
 \bbP\left(B^-_{v,t}\neq \cT^-_{v,t}(\bd^+)\right)&\leq \bbP\left(|\cT^-_{v,t}|> s\D^t \right)
 + \bbP\left(B^-_{v,t}\neq \cT^-_{v,t}(\bd^+); |\cT^-_{v,t}|\leq s\D^t\right)\nonumber\\
& \leq A\,e^{-a\,s} + \sum_{\ell=1}^{s\D^t}  \bbP(G_v(\ell)\neq \cT^-_v(\ell)) \leq A\,e^{-a\,s} +\frac{s^3\D^{3t+2}}{n-s\D^t}\,.
 \end{align}
The estimate \eqref{eq:coup31} follows by taking $s=K\log n$ for some large enough constant $K$, and by taking $n$ sufficiently large. 
\end{proof}

\subsection{Proof of Lemma \ref{le:approx}}\label{sec:le:approx}
Recall that in both models DCM($\bd^\pm$) and OCM($\bd^+$) one has w.h.p.\ a unique stationary distribution for the simple random walk on $G$, which we denote $\pi_0$. The starting point is a result that follows directly from \cite{BCS1,BCS2}, which allows us to replace the unknown distribution $\pi_0$ with a local approximation.
\begin{proposition}\label{prop:bcsproxy}
For any fixed $\e>0$, taking $h=\e\tent$, as $n\to\infty$ both models satisfy
\begin{equation}\label{eq:approxa1}
\left\|\muin P^h - \pi_0\right\|_{\tv}\overset{\P}{\longrightarrow} 0.
		\end{equation}
\end{proposition}
\begin{proof}
For a specific choice of $\e=\e_0$, this result appears in \cite[Eq. (11)]{BCS1} for model 1 and \cite[Eq. (12)]{BCS2} for model 2. In fact, the proofs in \cite{BCS1,BCS2} apply to any fixed $\e\in(0,\e_0)$ without modifications. Since $\left\|\muin P^h - \pi_0\right\|_{\tv}$  is monotone in $h$ the statement \eqref{eq:approxa1} holds for all $\e>0$.
\end{proof}
To prove Lemma \ref{le:approx}, by monotonicity of $\left\|\l P^t - \pi_0\right\|_{\tv}$ as a function of $t$, we may restrict to sequences $t=t(n)\to\infty$ with $t=o(\log n)$. Thus, taking advantage of Proposition \ref{prop:bcsproxy}, the conclusion of Lemma \ref{le:approx} is a consequence of the following result.
\begin{proposition}\label{prop:prox}
There exists $\e>0$ such that if $h=\e\tent$, then for any  $t=t(n)\to\infty$ with $t=o(\log n)$, for any widespread measure $\l$: 
\begin{equation}\label{eq:prox1}
\left\|\l P^t-\muin P^h\right\|_{\tv}\overset{\P}{\longrightarrow} 0.
		\end{equation}
\end{proposition}
\begin{proof}
The proof is based on a first moment argument.  Indeed, it suffices to 
show that
\begin{equation}
\label{eq:expe}
\lim_{n\to\infty}\bbE\left[\left\|\l P^t-\muin P^h\right\|_{\tv}\right] =0. 
\end{equation}
Observe that
\begin{align}\label{eq:expe1}
&\bbE\left[\left\|\l P^t-\muin P^h\right\|_{\tv}\right] 
=\frac{1}{2}\sum_{j\in[n]}\bbE\[\left|\l P^t(j)-\muin P^h(j)\right|\]
\nonumber
\\
&\qquad \qquad \leq \frac12\bbE\[\left|n\l P^t(\cI)-n\muin P^t(\cI)\right|\]]+\frac12\bbE\[\left|n\muin P^t(\cI)-n\muin P^h(\cI)\right|\],
\end{align}
where $\cI$ denotes an independent  uniformly random vertex in $[n]$ and the expectation $\bbE$ is understood to include the expectation over $\cI$ as well. Consider the first term above. We are going to use Lemma \ref{le:coup1} for model 1 and Lemma \ref{le:coup3} for model 2. Notice that since these estimates apply to any fixed vertex $v$, they apply just as well if the vertex $v$ is taken to be uniformly random in $[n]$, i.e.\ if $v=\cI$ as it is the case here. In particular,  since $t=o(\log n)$, as $n\to\infty$, 
\begin{align}\label{eq:coup0331}
\bbP\left(B^-_{\cI,t}\neq \cT^{-}_t\right)\to 0,
\end{align}
where we use the unified notation $\cT^{-}_t$ for  the first $t$ generations of the tree $\cT^-_{\cI}$ in either model 1 or model 2.
Next, note that by definition, if $B^-_{\cI,t}=\cT^{-}_t$, then $$n\l P^t(\cI)-n\muin P^t(\cI)=nX_t(\l)-M_t\,,$$
where we use the notation from \eqref{eq:proc1} and \eqref{eq:defmar}. Therefore,  
\begin{align}\label{eq:bbbb}
\bbE\left[\left|n\l P^t(\cI)-n\muin P^t(\cI)\right|\right]
&\leq2\bbP\left(B^-_{\cI,t}\neq \cT^{-}_t\right) + \bbE\left[\left|M_t-nX_t(\l)\right|\right],
\end{align}
where we used the fact that 
\begin{align*}
\bbE\left[\left|n\l P^t(\cI)-n\muin P^t(\cI)\right|\,\big|\, B^-_{\cI,t}\neq \cT^{-}_t\right]
\leq 2,
\end{align*}
which follows from $\left\|\l P^t-\muin P^t\right\|_{\tv}\leq 1$.
Using Schwarz' inequality and Proposition \ref{prop:mara} it follows that 
\begin{align*}
\bbE\left[\left|M_t-nX_t(\l)\right|\right]^2\leq \g(\l)\r^{t}.
\end{align*}
Since $t=t(n)\to\infty$ as $n\to\infty$ and $\r\in(0,1)$, using  \eqref{eq:coup0331} we conclude that 
\begin{align*}
\lim_{n\to\infty} 	\bbE\left[\left|n\l P^t(\cI)-n\muin P^t(\cI)\right|\right]
=0,				\end{align*}
for all widespread measure $\l$. This settles the convergence of the first term in \eqref{eq:expe1}. To handle the second term, reasoning as in \eqref{eq:bbbb} we obtain
\begin{align*}
\bbE\left[\left|n\muin P^t(\cI)-n\muin P^h(\cI)\right|\right]
&\leq2 \bbP\left(B^-_{\cI,h}\neq \cT^{-}_h\right) + \bbE\left[\left|M_t-M_h\right|\right]
\end{align*}
If $h\leq \frac{\log n}{4\log \D}$,  Lemma \ref{le:coup1} and Lemma \ref{le:coup3} imply that both models satisfy 
\begin{align}\label{eq:coup332}
\bbP\left(B^-_{\cI,h}\neq \cT^{-}_h\right)\to 0.
\end{align}
Moreover, Schwarz' inequality, Proposition \ref{prop:marconv} and standard facts about square integrable martingales (see, e.g., \cite[Ch. 12]{Williams}) imply 
\begin{align*}
\bbE\left[\left|M_t-M_h\right|\right]^2&\leq \bbE[(M_t-M_h)^2]
 \leq \bbE[(M_t-M_\infty)^2] = C\r^{t}.
\end{align*}
Since the constant $C$ is bounded, letting $n\to\infty$ concludes the proof.

  \end{proof}

\bigskip

\subsection*{Acknowledgments}
We acknowledge support of PRIN 2015 5PAWZB ``Large Scale Random Structures", and of INdAM-GNAMPA Project 2019 ``Markov chains and games on networks''.
\bigskip

\bibliography{bibpagerank}
\bibliographystyle{plain}

\end{document}